\documentclass[12pt, a4paper]{amsart}

\usepackage[english]{babel}
\usepackage[utf8x]{inputenc}
\usepackage[T1]{fontenc}
\usepackage{amsfonts}
\usepackage{amsthm}
\usepackage{amssymb}
\usepackage{amsrefs}
\usepackage{tikz-cd}
\usepackage{tikz}
\usepackage{scalerel}
\usepackage{mathbbol}
\usepackage{slashed} 

\tikzset{
commutative diagrams/.cd,
arrow style=tikz,
diagrams={>=latex}}
\tikzcdset{  crossing over/.style={
/tikz/preaction={
/tikz/draw,
/tikz/color=\pgfkeysvalueof{/tikz/commutative diagrams/background color},
/tikz/arrows=-,
/tikz/line width=\pgfkeysvalueof{/tikz/commutative
diagrams/crossing over clearance}}}}

\usetikzlibrary{arrows}

\usepackage[a4paper,top=3cm,bottom=3cm,left=3cm,right=3cm,marginparwidth=1.75cm]{geometry}

\usepackage{amsmath}
\usepackage{amsmath,tikz-cd}
\usepackage{graphicx}
\usepackage[colorinlistoftodos]{todonotes}
\usepackage[colorlinks=true, allcolors=blue]{hyperref}

\theoremstyle{plain} \newtheorem{theorem}{Theorem}[section]
\theoremstyle{plain} \newtheorem{lem}[theorem]{Lemma}
\theoremstyle{plain} \newtheorem{prop}[theorem]{Proposition}
\theoremstyle{plain} \newtheorem{cor}[theorem]{Corollary}
\theoremstyle{definition} 
\theoremstyle{definition}\newtheorem{ex}[theorem]{Example}
\theoremstyle{definition}\newtheorem{comm}[theorem]{Comment}

\newtheorem*{theorem*}{Theorem}
\newtheorem*{definition*}{Definition}
\newtheorem*{lem*}{Lemma}

\newcommand{\thmref}[1]{Theorem~\ref{#1}}
\newcommand{\propref}[1]{Proposition~\ref{#1}}
\newcommand{\ppropref}[1]{Prop~\ref{#1}}

\newcommand{\lemref}[1]{Lemma~\ref{#1}}

\numberwithin{equation}{section}

\renewcommand{\t}{\tau}

\newcommand{\ul}{\underline}

\newcommand{\mto}{\mapsto}

\newcommand{\x}{\times}
\newcommand{\ox}{\otimes}

\newcommand{\aand}{\mbox{and}}

\newcommand{\ffor}{\mbox{for}}

\newcommand{\iin}{\mbox{in}}

\renewcommand{\to}{\rightarrow}

\newcommand{\too}{\longrightarrow}
\newcommand{\Too}{\Longrightarrow}

\newcommand{\vtwo}{\vskip 2mm}
\newcommand{\vthree}{\vskip 3mm}
\newcommand{\vfour}{\vskip 4mm}

\newcommand{\hthree}{\hskip 3mm}

\newcommand{\hfive}{\hskip 5mm}
\newcommand{\hten}{\hskip 5mm}

\renewcommand{\r}{\right}
\renewcommand{\l}{\left}

\newcommand{\<}{\subset}
\newcommand{\ii}{^{-1}}
\newcommand{\pr}{^{\prime}}

\newcommand{\s}{\sigma}

\renewcommand{\a}{\alpha}
\renewcommand{\b}{\beta}

\newcommand{\la}{\lambda}
\newcommand{\om}{\omega}

\renewcommand{\L}{\Lambda}
\renewcommand{\d}{\delta}

\newcommand{\oo}{\infty}

\newcommand{\pd}{\partial}
\newcommand{\ch}{\textsf{ch}}
\renewcommand{\c}{\textsf{c}}
\newcommand{\p}{\textsf{p}}
\newcommand{\sw}{\textsf{sw}}
\newcommand{\cf}{\textsf{cf}}
\newcommand{\pf}{\textsf{cf}}
\newcommand{\swf}{\textsf{cf}}
\newcommand{\Ahat}{\widehat{\textsf{A}}}
\newcommand{\Td}{\textsf{Td}}

\newcommand{\Ker}{\textmd{\small {\rm Ker}}\,}

\newcommand{\Si}{\Sigma}

\usepackage{dsfont}
\usepackage{relsize}

\newcommand{\rk}{\mathrm{rk}}

\newcommand{\Ss}{\mathcal{S}} 
 
\newcommand{\Ee}{\mathcal{E}}     
    
\newcommand{\Vv}{\mathcal{V}}    
    
\newcommand{\Oo}{\mathcal{O}}

\newcommand{\Z}{\mathbb{Z}}                            
\newcommand{\R}{\mathbb{R}}   
\newcommand{\C}{\mathbb{C}}   
\newcommand{\N}{\mathbb{N}} 
\newcommand{\Q}{\mathbb{Q}} 
\newcommand{\E}{\mathbb{E}}

\newcommand{\Sbb}{\mathbb{S}}

\newcommand{\Gl}{\mathrm{Gl}} 

\newcommand{\hh}{\mathbb{h}}
\newcommand{\kk}{\mathbb{k}}

\renewcommand{\x}{\times} 

\newcommand{\Det }{\mathrm{Det}\,}

\newcommand{\etaSOX}{\mathlarger{\mathlarger{\eta}}\mbox{{\small SO}}_{\mbox{\tiny X}}}
\newcommand{\etaSOXq}{\mathlarger{\mathlarger{\eta}}\mbox{{\small SO}}_{\mbox{\tiny X,q}}}

\newcommand{\nuX}{\nu_{\mbox{\tiny X}}}

\newcommand{\etaaX}{\mathlarger{\mathlarger{\mathlarger{\eta}}}\mbox{{\small O}}_{\mbox{\tiny X}}}

\newcommand{\etaaSOX}{\mathlarger{\mathlarger{\mathlarger{\eta}}} \mbox{{\small SO}}_{\mbox{\tiny X}}}

\newcommand{\etaaUXq}{\mathlarger{\mathlarger{\mathlarger{\eta}}} \mbox{{\small U}}_{\mbox{\tiny X,q}}}

\renewcommand{\H}{\mbox{{\small {\rm H}}}}
\newcommand{\K}{\mbox{{\small {\rm K}}}}
\newcommand{\KO}{\mbox{{\small {\rm KO}}}}
\newcommand{\Hss}{{\scaleto{{\rm {\bf \, H}}}{3.25pt}}}
\newcommand{\Hsss}{{\scaleto{{\rm {\bf \, H\Z_2}}}{3.25pt}}}
\newcommand{\Kss}{{\scaleto{{\rm {\bf \, K}}}{3.25pt}}}
\newcommand{\KOss}{{\scaleto{{\rm {\bf \, KO}}}{3.25pt}}}
\newcommand{\MSp}{\mbox{{\small {\rm MSp}}}}
\newcommand{\MSO}{\mbox{{\small {\rm MSO}}}}
\newcommand{\MU}{\mbox{{\small  {\rm MU}}}}
\newcommand{\MO}{\mbox{{\small {\rm MO}}}}
\newcommand{\BSO}{\mbox{{\small {\rm BSO}}}}
\newcommand{\BU}{\mbox{{\small {\rm BU}}}}

\newcommand{\mso}{{\scaleto{{\rm {\bf \,MSO}}}{3.25pt}}}
\newcommand{\muu}{{\scaleto{{\rm {\bf \,MU}}}{3.25pt}}}
\newcommand{\mo}{{\scaleto{{\rm {\bf \,MO}}}{3.25pt}}}

\newcommand{\so}{{\scaleto{{\rm {\bf SO}}}{3.25pt}}}
\newcommand{\uu }{{\scaleto{{\rm {\bf U}}}{3.25pt}}}

\newcommand{\xX}{\times_{\mbox{\tiny X}}}

\newcommand{\uX}{\sqcup_{\mbox{\tiny X}}}

\newcommand{\Tx}{T_{\mathsmaller{X}}}

\newcommand{\Tm}{T_{\mathsmaller{M}}}
\newcommand{\Tn}{T_{\mathsmaller{N}}}
\newcommand{\Tw}{T_{\mathsmaller{W}}}

\newcommand{\Tmn}{T_{\mathsmaller{M\times_{\mathsmaller{{\tiny X}}}N}}}
\newcommand{\MxN}{M\times_{\mathsmaller{X}}N}
\newcommand{\MxY}{M\times_{\mathsmaller{X}}Y}
\newcommand{\MxMpr}{M\times_{\mathsmaller{X}}M\pr}

\newcommand{\sww}{{\rm sw_1}}

\newcommand{\MX}{\mbox{$\mathsmaller{M/X}$}}

\begin{document}

\title{Vertical Genera}

\author{Niccol\`{o} Salvatori and Simon Scott}
\date{}


\maketitle

\section{Introduction}

In this paper we extend the classical constructions of genera on a compact boundaryless manifold, such as the signature, $\hat{A}$-genus and Todd genus, to the case of families of manifolds parametrised by a space $X$.\\

To do so requires the construction of certain generalised Pontryagin and Chern classes 
as maps from bordism cohomology to the singular cohomology of $X$. These relate to the usual characteristic classes on $K$-theory via the Chern-Dold character. \\ 

Bordism theory and genera permeate  both algebraic topology and geometric index theory, with broad implications for theoretical physics. Specifically, Classical (elliptic) genera define partition functions in  type II superstring theory, while the classical bordism ring defines in its `quantised' form the ontology of topological QFT.  The parametrised genera considered here may be relevant to  `families' versions of these applications; for instance, it is natural to contemplate  a fibred/vertical TQFT as a symmetric monoidal functor from the bordism category of fibre bundles  to the category of vector bundles. \\

We begin by briefly reviewing the basic constructions of classical bordism. 

\subsection{Classical genera}

A rational  genus with values in an integral domain $R$ over $\Q$  is a ring homomorphism 
\begin{equation}\label{varphi}
\varphi: \MSO_* \ox\Q \to R
\end{equation}
with $\MSO_*$ the oriented bordism ring. An element of $\MSO_*$ is the equivalence class $[N]$ of a closed oriented $n$-manifold $N$ in which $N\sim N'$  (bordant) if $N \sqcup -N' = \partial W$ is the boundary of a compact $(n+1)$-manifold $W$, with $-N'$ the manifold $N'$ with orientation reversed. The additive and multiplicative structure of $\MSO_*$ is defined by disjoint union and direct product, respectively. Thus, $\varphi$ assigns to each oriented closed manifold $N$ a element  $\varphi(N)\in R\ox \Q$ with  
$\varphi(N\sqcup N')=\varphi(N)+\varphi(N')$
and $\varphi(N\x N') = \varphi(N)\varphi(N')$, and such that if $N = \partial W$ is a boundary  then  $\varphi(N)=0$. These properties are captured and characterised in a more granular way using the Pontryagin numbers 
\begin{equation}\label{Pnumbers}
p_J(N) = \langle p_{j_1}\cdots p_{j_r}, [N] \rangle  = \int_N p_{j_1}\cdots p_{j_r} \in\Q
\end{equation}
defined for each partition $J=(j_1,\dots,j_r)$ of $\frac{1}{4}\dim N  \in \N$, with $p_k\in H^{4k}(N,\Q)$ the $k$th Pontryagin class of $N$, along with  a formal power series 
\begin{equation}\label{fphi}
f_\varphi  \in R\ox\Q[[x]]
\end{equation}
prescribing how to construct the genus $\varphi$ from linear combinations of the $p_J$. 
The intrinsic numbers $p_J(N)$  are cobordism invariants, in fact
\begin{equation}\label{classical pontryagin}
[N] = [N']   \ \ \ \iin\ \MSO_*\otimes \Q  \  \ \iff  \ \  p_J(N) = p_J(N') \ \ \forall \ |J| =\dim N/4,
\end{equation}
equivalently,  $p_{j_1}(N)\cup \cdots \cup p_{j_r}(N)$ with $j_1\leq j_2\leq \ldots \leq j_r$ and $\sum j_r = \dim N/4$ are a basis for the dual cobordism ring $\MSO^*$. The characteristic power series $f_\phi$ has the form
\begin{equation}\label{Pnumbers}
f_\varphi (x) = \frac{x}{e_\varphi(x)}  \in R\ox\Q[[x]]
\end{equation}
in which  $e_\varphi  \in R\ox\Q[[x]]$, with leading term $x$ and $e_\varphi(-x) = -e_\varphi(x)$, is the formal inverse power series to   the logarithm  series associated to $\varphi$ defined by
$$l_\varphi (x)   := \sum_{n=1}^\oo \frac{\varphi(\C P^{2n})}{2n+1}\, x^{2n+1} \ \in R\ox\Q[[x]].$$
To give $l_\varphi$, or $e_\varphi$ or $f_\varphi $, is the same thing as $\varphi$, as follows from Thom's identification $\MSO_* \ox\Q \cong \Q[\C P^2, \C P^4, \C P^6, \ldots ],$ i.e. the spaces $\C P^2, \C P^4, \C P^6, \ldots$ form a generating set for $\MSO_* \ox\Q$; so, for example, $\MSO_2 \ox\Q$ is generated by $\C P^2$ while $\MSO_4 \ox\Q$ is generated by $\C P^2 \x \C P^2$ and $\C P^4$.
Explicitly, the product $\Pi_{i=1}^N f_\varphi(x_i)$ is symmetric and even in each variable $x_i$ and so each summand of homogeneity degree $k$ can be written as a polynomial in the elementary symmetric functions of $x^2_1, \ldots, x_N^2$. On taking for  $x_i$ the Chern class of the $i$th splitting line bundle of the complexified tangent bundle, those elementary symmetric functions become the Pontryagin classes $p_k=p_k(TN)$. The component in $H^{\dim N}(N,\Q)$ of the  product  gives in this way a polynomial $K(p_1, \ldots, p_k)$ and by a theorem of Hirzebruch
$$ \varphi(N) = \langle K(p_1, \ldots, p_k), [N] \rangle .$$

The most familiar examples are the signature of a $n=4k$ dimensional manifold for which $e_\varphi(x) = \tanh(x)$, and the $\hat{A}$-genus for which $e_\varphi(x) = 2\sinh(x/2)$. More recent work has been focused around the Witten genus  $\omega: \MSO_* \ox\Q \to \Q[[q]]$ for which $e_\varphi(x) = 2\sinh(x/2) \Pi_{n\geq 1} (1 - q^n e^x)(1 - q^n e^{-x})(1 + q^n)^{-2},$ and its applications to modular forms, M-theory, and elliptic cohomology.  \\

Such genera are natural evolutions of classical Riemann-Roch type formulae of algebraic geometry. This is made precise by  geometric and topological index theory and the Atiyah-Singer index theorems. In the above cases, the signature is the index of the operator $d+ d^* : \Omega^+(N) \to \Omega^-(N)$ on the even component of the Hodge decomposition of the de Rham complex, 
while if $N$ is a spin manifold the $\hat{A}$ genus is equal to the index of the Dirac operator $\slashed{\partial} : \Ss^+(N) \to \Ss^-(N)$ on positive spinors. The Witten genus is formally the index of a Dirac operator on the loop space of a $4k$ dimensional spin manifold;  a rigorous construction of this operator is missing, but there is a well developed stable homotopy description of the Witten genus as map of spectra  from elliptic cohomology to topological modular forms  \cite{Hopkins}.\\

\subsection{Vertical genera}  The most geometrically appealing case of the scenario we consider here is for a fibre bundle defined by a submersion $\pi: M\to X$ endowed with a choice of orientation  on the vertical tangent bundle $T_\pi$ over $M$  (along the fibres), specifying a smooth family of closed oriented-diffeomorphic manifolds $N_x = \pi\ii(x) $, note $M$ need not be orientable. A second such bundle $\pi\pr: M\pr\to X$ is fibrewise bordant to $\pi$ if there is a vertically oriented (see below) manifold $W$ with boundary diffeomorphic to $M\sqcup -M\pr$ and a map $\s:W\to X$ (we do not require $\s$ to be a submersion) which restricts to $\pi$ on $M$ and $\pi\pr$ on $M\pr$. 
 The set of all such objects forms a graded ring $\etaaSOX = \bigoplus_q \etaSOXq$ summed over fibre dimension $q$ in which the additive and multiplicative structure is defined by fibrewise disjoint union and fibre product. As a natural extension of \eqref{varphi},  a fibred (vertical) rational genus is defined to be a ring homomorphism 
\begin{equation}\label{varphiX}
\etaSOX \ox \Q \to H^*(X,\Q).
\end{equation} 
$\etaaSOX$ reduces when $X=pt$ is a point to  the classical bordism ring $\MSO_*$ while \eqref{varphiX} with $R=\Q$ then reduces to \eqref{varphi}.  An indication of the likely form of such genera is inferred from the cohomological version  of the Atiyah-Singer index Theorem for a smooth family $\mathcal{D}^\text{spin}_\pi$ of Dirac operators defined by a geometric spin fibration $\pi: M \rightarrow X$ (submersion) with index bundle $\text{Ind}\,\mathcal{D}^\text{spin}_\pi \in \K^0(X)$, which states that
\begin{equation}\label{FIT}
\text{ch}(\text{Ind}\,\mathcal{D}^\text{spin}_\pi) = \hat{A}(\pi) := \pi_!( \hat{A}(T_\pi)).
\end{equation} 
Here, $\pi_!( \hat{A}(T_\pi))$ -- a vertical genus -- is the the push-forward (integral over the fibre) to $H^*(X,\Q)$ of the $\hat{A}$-genus in $H^*(M,\Q)$ of the vertical tangent bundle $T_\pi$. \\  

The construction of vertical genera does not, in fact, depend in any essential way on $\pi$ being a submersion, a similar construction can be made for any homotopy class of maps $f: M \to X$ from a  closed {\it vertically oriented} manifold $M$ of dimension $k$  to $X$, defining a bordism homology class $[f]\in \MSO_k(X)$ (see \S2 for precise definitions). The natural inclusion 
\begin{equation}\label{etaX to MSOX homology}
\etaSOXq \to \MSO_{\dim X + q}(X)
\end{equation}
is then dual by Atiyah-Poincar\'{e} duality to an inclusion   
\begin{equation}\label{etaX to MSOX}
\etaSOXq \to \MSO^{-q}(X)
\end{equation}
identifying $\etaaSOX$ as a graded subring of bordism cohomology $\MSO^*(X)$. (The (co)bordism cohomology groups $\MSO^r(X)$ are defined for all integers $r$,  non-trivially so for negative integers,  but $\MSO^r(X)=0$ for $r>\dim X$.) The upshot is that $\etaaSOX$ may be viewed either homologically or dually as a cohomology theory. For ring properties this is important since the graded multiplication defined by the fibre product $$ \mathlarger{\mathlarger{\eta}}\mbox{{\small SO}}_{\mbox{\tiny X,p}}\x  \mathlarger{\mathlarger{\eta}} \mbox{{\small SO}}_{\mbox{\tiny X,q}} \to \mathlarger{\mathlarger{\eta}} \mbox{{\small SO}}_{\mbox{\tiny X,p+ q}}$$ is seen to coincide via \eqref{etaX to MSOX} with the cup product on  $\MSO^*(X)$.
 \\

Consequently, the general definition of an oriented vertical genus over a smooth closed manifold $X$  is a ring homomorphism 
\begin{equation*}
\nuX :  \MSO^*(X) \ox\Q \to \H^*(X, \Q)
\end{equation*}
from the oriented cobordism cohomology ring $\MSO^*(X) $ to the singular cohomology ring of $X$.  Our purpose here is to give a construction of such genera in terms of  vertical Pontryagin classes
\begin{equation*}
\p^{\mso}_J : \MSO^*(X) \to \H^*(X, \Z), \hten J = (j_1,\ldots, j_m)\in \N^\oo,
\end{equation*}
 which collectively  characterise $\MSO^*(X)\ox\Q $   in the same way that, classically \eqref{classical pontryagin}, Pontryagin numbers characterise  the classical oriented (co)bordism ring $\MSO_*\ox\Q =\MSO(pt)\ox \Q$ 
 coefficient ring to $\MSO^*(X)\otimes \Q$. Precisely, extending \eqref{classical pontryagin}, one has: \\

{\bf Theorem 1.}   $\a=\a\pr \ \,\text{in}\ \MSO^*(X)\otimes \Q  \, \Leftrightarrow \,  \p^{\mso}_J(\a)=\p^{\mso}_J(\a\pr)\  \ {\rm  in} \   \H^*(X, \Q) \  \forall\ J \subset \mathbb{N}^\infty.$\\

A corresponding result for complex bordism cohomology $\MU^*(X)$ holds in terms of vertical Chern classes
$$\c^{\muu}_J : \MU^*(X) \to \H^*(X, \Z)$$
which relate  to the Conner-Floyd Chern classes in $\MU^*(X)$ via the Chern-Dold character 
$\ch^{\muu} : \MU^*(X) \to \H^*(X, \MSO_*)$.  Vertical Chern classes
can be built  into stably complex vertical  genera, defining ring homomorphisms $ \MU^*(X) \ox\Q \to \H^*(X, \Q).$\\ 

For unoriented bordism cohomology  the correspondence holds in terms of vertical Stiefel-Whitney classes ${\rm sw}^{\mo}_J : \MO^*(X) \to \H^*(X, \Z_2)$. \\ 

Vertical genera on $\MSO^*(X)\ox\Q $, or in their geometrically natural habitat on  $\etaaSOX$,  admit a straightforward characterisation mirroring the classical case: \\

{\bf Theorem 2.} \ \ {\it To each characteristic power series $f_\nu \in \mathbb{Q}[x]$, and resulting multiplicative sequence, there is a vertical genus 
$\nu_{\mbox{\tiny X}} :  \MSO^*(X)\ox\Q \to \H^*(X, \Q)$ given by a power series in the vertical Pontryagin classes $\p^{\mso}_J$.} \\

An analogous statement holds for vertical stably complex genera. Each classical genus is thus seen to admit a vertical/fibrewise analogue. For example, 
one has the vertical $\hat{A}$-genus $\a\mto \hat{A}(\a)$ with component polynomials 
$$\hat{A}_1\left(\a\right)= -\frac{1}{24}\p^{\mso}_1(\a), \ \ 
\hat{A}_2\left(\a \right)= \frac{7}{2^7 \cdot 3^2 \cdot 5}\left(-4 \p^{\mso}_2(\a)+7\p^{\mso}_1(\a)^2\right), \ldots \ \ \in \H^*(X,\Q)$$ 
and so on. Theorem 2 says that familiy's index has the following multiplicativity property:
\begin{cor}
 In $\H^*(X, \mathbb{Q})$ one has
\begin{align}\label{(0.22)}
\hat{A}(\pi\times_{\mbox{\tiny X}} \pi\pr)=\hat{A}(\pi) \cup \hat{A}(\pi\pr)
\end{align}
for $\pi\times_{\mbox{\tiny X}} \pi\pr : M\xX M\pr\to X$  the fibre product (with fibre $\pi\ii(x)\x (\pi\pr)\ii(x)$ at $x$).  
\end{cor}
Further, since rationally the Chern character $\rm{ch}$ is a ring isomorphism, the classical cobordism invariance of the pointwise index for a single manifold is repositioned for fibrations by Theorems 1 and 2 as:\\

\begin{cor} {\rm (Vertical cobordism invariance of the family index)}
 If $\pi: M \rightarrow X$ and $\pi\pr: M' \rightarrow X$ are vertically bordant spin fibrations, so $[\pi]=[\pi']$, then
\begin{align*}
\rm{Ind}\,\mathcal{D}^\text{spin}_\pi=\rm{Ind}\,\mathcal{D}^\text{spin}_{\pi\pr} \qquad \text{ in } \K(X) \otimes \mathbb{Q}.
\end{align*}
If $\pi: M \rightarrow X$ is a boundary fibration, then $\rm{Ind}\,\mathcal{D}^\text{spin}_\pi=0$.
Moreover, for any $\pi, \pi\pr$ 
$$\rm{Ind}\,\mathcal{D}^\text{spin}_{\pi\times_{\mbox{\tiny X}} \pi\pr}= \rm{Ind}\,\mathcal{D}^\text{spin}_\pi \otimes \rm{Ind}\,\mathcal{D}^\text{spin}_{\pi\pr}.
$$
\end{cor}

Similarly, there is a vertical  $L$-genus which  from \cite{Atiyah2}  implies the vertical bordism invariance of a natural geometric family of signature operators $\mathcal{D}^\text{sig}_\pi$ and the multiplicativity of its index bundle.  It would be interesting to know more about other vertical elliptic genera from the viewpoint here, such as the vertical Witten genus taking values in $H^*(X, \Q[[q]])$, which for a fibre bundle $M \stackrel{\pi}{\rightarrow} X$ of $4n$-dimensional closed spin manifolds  may, conjecturally,  be close to the Chern character of the formal index bundle of the resulting family of Witten-type Dirac operators on the fibrewise loop space of  $M$.

\section{Vertically oriented bordism homology}

All manifolds will be smooth and compact. The tangent bundle of a manifold $N$ will be denoted $\Tn$. The orientation bundle $ \mathcal{O}_N$ of $N$ is the principal  $\Gl(1,\R)$  frame bundle $F(\Det \Tn)$  of the determinant line bundle $\Det \Tn := \wedge^n \Tn$, or, equivalently for our purposes, the   principal $\Z_2$ bundle $F(\Det \Tn)_{\Gl(1,\R)}\x\Z_2$ coming from the map $\Gl(1,\R)\to \Z_2$, $x\mto x/|x|$, whose transition functions are $\det(g_{ij})/|\det(g_{ij})|$ with $g_{ij}$ the transition functions of $\Tn$.  More generally, the orientation bundle $ \mathcal{O}_{E-F}$ of a stable bundle $E- F$ in the $\K$-theory ring of real vector bundles $\KO(X)$ is the  frame bundle of the determinant line bundle $\Det E\ox \Det F^*$, and a choice of orientation on $E- F$, if one exists, is a choice of trivialisation of $ \mathcal{O}_{E-F}$. \\

A smooth map $f:N\to X$ is said to be vertically oriented if there exists an isomorphism of orientation bundles 
$\mathcal{O}_N \cong f^* \mathcal{O}_X$ with a choice of isomorphism $ \a_{\mbox{\tiny N,X}} : \mathcal{O}_N \to f^* \mathcal{O}_X$ specified. One has:\\ 
\begin{lem} \label{v-orient}
A vertical orientation on $f:N\to X$ is equivalent to a choice of orientation on its stable normal bundle $$\Vv^{\rm st}_f := f^*\Tx - \Tn \ \in \KO(N).$$ 
\end{lem}
\begin{proof}
$f^*\mathcal{O} _X = \mathcal{O}_{f^* \Tx}$ and so a vertical orientation on $f$ is equivalent to a choice of  isomorphism $\Det \Tn \cong f^*\Det  \Tx \cong \Det f^* \Tx$, that is, a trivialisation of the line bundle $\Det \Tn \ox \Det f^* \Tx$, which is the same as a trivialisation of its frame bundle. 
\end{proof}

The induced orientation in a fibre product is defined as follows. Let $f:M\to X$ and $g:N\to X$ be smooth maps. Then there is a commutative diagram
\begin{equation}\label{fibre prod diag}
\begin{tikzcd}
\MxN \arrow{r}{\mu} \arrow[swap]{dr}[above]{\hfive \   s = f \times_{{\tiny X}} g}\arrow{d}{\nu} &  M\arrow{d}{f} \\
N \arrow{r}{g}   & X
\end{tikzcd}
\end{equation}
in which
$$ \MxN =\{(y,z)\in M\x N \ | \ f(y) = g(z)\}$$
is the fibre product of $(M,f)$ and $(N,g)$. Equivalently, $ \MxN = (f \x g)^{-1}(\text{diag}(X))$ with  $\text{diag}(X)$ the diagonal in $X\x X$.  The maps  $\mu, \nu$ are the projection maps. The diagonal map is, thus,
\begin{equation}\label{sfpgq}
s= f\circ \mu = g\circ \nu.
\end{equation}
 $f$ and $g$ are  transverse if for $(y,z)\in M \xX N$ one has $df_y(T_y M) + dg_z(T_z N) = T_x X$, where $x =f(y) = g(z)$, or,  equivalently, if $f \x g$ is transverse to $\text{diag}(X)$. If $f$ and $g$ are transverse maps then $\MxN$ is a smooth manifold of dimension $\dim M + \dim N - \dim X$; if $\dim X > \dim M + \dim N$ then $\MxN$ is the empty manifold (see Prop II.4 \cite{Lang}).

\begin{prop}\label{induced orientations}
Let $f$ and $g$ be transverse maps. Then 
\begin{equation}\label{Vv pull back}
\Vv^{\rm st}_\nu  = \mu^* \Vv^{\rm st}_f 
\end{equation} and
\begin{equation}\label{Vv product}
\Vv^{\rm st}_{f \times_{{\tiny X}} g}  = \mu^* \Vv^{\rm st}_f + \nu^* \Vv^{\rm st}_g.
\end{equation}
\vtwo
The fibre product map $f \times_{{\tiny X}} g: \MxN\to X$ is canonically vertically oriented by vertical orientations on $f:M\to X$ and $g:N\to X$. The pull-back map 
$$ \nu: g^*(M\stackrel{f}{\to}X) := \MxN \to N$$
is canonically vertically oriented by a vertical orientation on $f:M\to X$.
\end{prop}
\begin{proof}
For transverse $f$ and $g$ there is a bundle isomorphism 
\begin{equation}\label{Tfibreprod}
\Tmn +s^* \Tx \cong \mu^* \Tm + \nu^* \Tn,
\end{equation}
as follows from the exact sequence \cite{Joyce}
$$0\to \Tmn \stackrel{d\mu\oplus d\nu}{\too} \mu^* \Tm + \nu^* \Tn \stackrel{\mu^* df - \nu^* dg}{\too} s^* (\Tx) \to 0.$$
Exactness is by \eqref{sfpgq} and transversality. Hence as stable bundles
$$\Tmn - \nu^* \Tn  \stackrel{\eqref{Tfibreprod}}{=}    \mu^* \Tm - s^* \Tx   =    \mu^* \Tm - (f\circ \mu)^* \Tx  =   \mu^* (\Tm - f^* \Tx),$$
which is \eqref{Vv pull back}, and
\begin{equation}\label{Vvs}
\Tmn - s^* \Tx  =    (\mu^*\Tm - s^* \Tx) + (\nu^* \Tn- s^* \Tx)
\end{equation}
which from \eqref{sfpgq} is \eqref{Vv product}. By \lemref{v-orient} the stated induced orientations are immediate from taking top exterior powers (determinant bundles) of these identifications.  
\end{proof}

\vthree

From Atiyah's construction \cite{Atiyah1}, an element $[f,N,\mu]\in \MSO_n(X)$ of bordism homology is represented by a vertically oriented map $f:N\to X$ from a closed manifold $N$ of dimension $n$ --- thus with a given orientation bundle isomorphism  $\mu: f^*\mathcal{O}_X\cong \mathcal{O}_N$ ---  with two such triples $(f,N,\mu)$, $(f\pr,N\pr,\mu\pr)$, being equivalent  if there is a triple $(\s,W,\la)$ with  $\s:W\to X$ an oriented map  from a manifold $W$ of dimension $n+1$ with boundary $\partial W\cong N\sqcup N\pr$ such that $\s_{|N} = f,$ $\s_{|N\pr} = f\pr$, and $\la: \s^*\mathcal{O}_X\cong \mathcal{O}_W$ with $\la_{|N} = \mu,$ $\la_{|N\pr} = -\mu\pr$. 
\begin{comm}\label{atiyah}
In fact, Atiyah defines a bordism homology group $ \MSO_n(X, \a)$ for each principal $\Z_2$-bundle $\a$ on $X$. In that notation, $ \MSO_n(X)$ is the group  $\MSO_n(X,\Oo_X)$ --- note, in contrast to its usage here,  in \cite{Atiyah1}  `$\MSO_n(X)$' refers to the case  $\Oo_X = X\x\Z_2$, i.e. orientable $X$.
\end{comm}
The additive abelian structure of $ \MSO_n(X)$  is defined by the fibrewise disjoint union $(f\uX f\pr ,N\uX N\pr, \a\uX\a\pr)$. Fibre product defines a natural product  
\begin{equation}\label{msoprod}
\MSO_m (X)  \times \MSO_n (X) \to  \MSO_{m+n -\mbox{\tiny {\rm dim} X}} (X),
\end{equation}
and for $\phi:Y\to X$ a Umkehr/Gysin/transfer  map  
\begin{equation}\label{msoumkehr}
\MSO_m (X)  \to  \MSO_{m+\mbox{\tiny {\rm dim}Y- \mbox{\tiny {\rm dim}X}}} (Y).
\end{equation}
The product \eqref{msoprod} maps $([f,M,\mu], [f\pr,M\pr,\mu\pr])$ to $ [ f\xX f\pr: \MxMpr,\mu\xX\mu\pr]$ 
where $f, f\pr$  are chosen (homotoped) within their bordism classes to be mutually transverse.   Similarly, \eqref{msoumkehr} is the map $[f,M,\a]\mto [f\xX \phi,\MxY,\phi^* \a]$ with $f$ chosen transverse to $\phi$.  $\mu\xX\mu\pr$ and $\phi^* \a$ indicate the canonical induced orientations of \propref{induced orientations}.  See \cite{Cohen} for more on Umkehr maps.\\

\eqref{msoprod} is not a graded product
\footnote{Except on the subgroup $\eta so_{X,q}$ of fibre bundles of fibre dimension  $q$ for which \eqref{msoprod} defines a graded product $\eta so_{X,q}\x \eta so_{X,q\pr} \to \eta so_{X,q+q\pr}$
coinciding with the cohomology product  via \eqref{etaX to MSOX}.}.
This is corrected  by  Atiyah-Poincar\'{e} duality which posits matters into  bordism cohomology, with respect to which the homology fibre product \eqref{msoprod} relates to the product on the graded bordism cohomology ring in the same way that the intersection product $\H_m(X) \x \H_n(X) \to \H_{m+n -\mbox{\tiny {\rm dim} X}} (X)$
on singular homology relates to the cup product via classical Poincar\'{e} duality. Precisely, one has a commutative array
\begin{equation}\label{MSOduality}
\begin{tikzcd}[column sep=scriptsize]
  \MSO^{\,-q} (X) \x  \MSO^{\,-q\pr} (X)   \arrow{r}     & \MSO^{\,-(q+q\pr)} (X) \\
 \MSO_{\mbox{\tiny {\rm dim} X + q}}(X) \x \MSO_{\mbox{\tiny {\rm dim} X + $q^{\pr}$}} (X) \arrow{r}   \arrow[u]  &  \MSO_{\mbox{\tiny {\rm dim} X + q + $q\pr$}}(X) \arrow[u] 
\end{tikzcd}
\end{equation}
in which the  duality map $\MSO_{\mbox{\tiny {\rm dim} X + q}}(X) \to \MSO^{\,-q} (X) $ is defined as follows. For vertically oriented $f:M\to X$ representing $[f, M, \a] \in \MSO_{\mbox{\tiny {\rm dim} X + q}} (X)$,  with orientation $\a$ on its stable normal bundle  
$$\Vv^{\rm st}_f  = f^* \Tx - \Tm$$
of rank
$$ {\rm rk}( \Vv^{\rm st}_f) = - q :=  - (\dim M  - \dim X),$$
we may choose an embedding 
$e_l : M \to \R^l$ and fibrewise it to the embedding 
$$e_l \x f  : M \to \R^l \x X,  \hten m\mto (e(m), f(m)),$$ 
which extends to an embedding of the normal bundle 
\begin{equation}\label{v(exf)}
\Vv_{e_l\x f} =  (e_l\x f)^*(\ul l  + \Tx)/\Tm \to \R^l \x X
\end{equation}
 of $e_l\x f$, with $\ul l$ the rank $l$ trivial bundle.  $\Vv_{e_l\x f}$ has rank 
$${\rm rk}(\Vv(e_l\x f)) = l-q.$$

The Thom construction collapses \eqref{v(exf)} contravariantly to the Pontryagin-Thom map 
\begin{equation}\label{PT}
\Si^l X_+ \to M^{\Vv_{e_l\x f}}.
\end{equation}
Here, $Y^V = {\rm Th}(V)$ is the Thom space (one-point compactification of the total space) of a vector bundle $V\to Y$; in particular, the iterated reduced suspension $\Si^l X_+$ is the Thom space of the trivial bundle $\ul l =\R^l \x X$.\\

Since $\Vv_{e_l\x f} - \ul l$ is a representative for the stable normal bundle $\Vv^{\rm st}_f $ this can be positioned invariantly.  $\Vv_{e_l\x f}$ is the same thing as the corresponding homotopy class of maps $M \to \BSO_{l-q}$ to the Grassmannian of oriented $l-q$ planes, covered by a map $\Vv_{e_l\x f} \to \xi_{l-q} $ to the universal bundle, 
giving an induced map of Thom pre-spectra $ M^{\Vv_{e_l\x f}} \to \MSO(l-q)$, where $\MSO(n) := \BSO_n^{\xi_n}$. With \eqref{PT} this defines an element of $[\Si^l X_+, \MSO(l-q)]$, which as a stable colimit is an element of 
$$\MSO^{-q}(X):=\lim_{\stackrel{\longrightarrow}{l}}[\Sigma^l X_+, \MSO(l-q)].$$
If we work in the stable category we may desuspend the map  \eqref{PT} directly to get the stable Pontryagin-Thom map 
\begin{equation}\label{stablePT}
X_+ \to \Si^{-l} M^{\Vv_{e_l\x f}}  = M^{\Vv_{e_l\x f} - \ul l } =  M^{\Vv^{\rm st}_f}
\end{equation}
to the Thom space of the rank $-q$ stable bundle $\Vv^{\rm st}_f$ and so an element in
$$ \MSO^{-q}(X) = [X_+, \MSO_{-q}],$$
where $\MSO_{-q}$ is the Thom spectrum (proper), i.e. the suspension spectrum  of the stable bundle in $\K(\BSO)$ of rank 0 which restricts on $\BSO_n$ to $\xi_n - \ul n$. 
As the graded ring product on $ \MSO^*(X)$ is induced by Whitney sum, the commutativity of  \eqref{MSOduality} follows from \eqref{Vv product}. The proof that  the homotopy class so constructed is independent of the bordism class of $f$ and that the assignment is bijective follows closely standard Thom constructions and the use of $L$-equivalence to mediate the duality map, similarly to \cite{Atiyah1}. \\

There is likewise a Pontryagin-Thom  map between the Thom complexes of the stable normal bundles
\begin{equation}\label{stablePT2}
X^{-\Tx}  \to   M^{-\Tm},
\end{equation}
which  is a particular instance of the extension to the case of a virtual bundle $\xi \to X$, for which one has  the Pontryagin-Thom map
\begin{equation*}
X^{\xi}  \to  M^{f^*\xi + \Vv^{\rm st}_f}
\end{equation*}
obtained by applying \eqref{PT} to the induced map $f^*\xi\to \xi$. Setting $\xi= \zeta - \Tx$ for a virtual bundle $\zeta$ on $X$, this can be put into the form  \cite{CrJa}, \cite{CoJo},
\begin{equation*}
X^{\zeta - \Tx}  \to  M^{f^*\zeta - \Tm}
\end{equation*}
which yields \eqref{stablePT} and \eqref{stablePT2} efficiently. \\

\vtwo

A similar characterisation holds for complex bordism homology $MU_m (X)$ comprising bordism classes $[M,f]$ of maps $f:M\to X$,  $\dim M =m$,  with a stable complex structure on  $\Vv_f^{\rm st} = f^*\Tx - \Tm$, meaning a class of complex structures on $f^*\Tx  + \Ee$ where $\Tm + \Ee = \ul n$, i.e. $\Vv_f^{\rm st}$ is in the image of the forgetful map $\K(X)\to \KO(X).$  There is a natural inclusion $\etaaUXq  \to \MU_{q+\dim X} (X)$. On the other hand, with $q=\dim M - \dim X$, we may choose an embedding $M\to \R^l$ with an isomorphism $\Vv_{e_l\x f} = \zeta_\R$ with $\zeta$  a stable complex vector bundle  of rank $l-q =2k$ identified with the pullback $\mu^*(\xi_k)$ of the canonical bundle $\xi_k\to \BU(k)$ for a homotopy class $\mu \to \BU(k)$. Taking Thom spaces  yields a map $\Sigma^l X_+\to M^{\Vv_{e_l\x f}} \to \MU(k)$ 
and so the duality isomorphism $\MU_q(X) \to \MU^{\dim X - q} (X)$ to the stable  homotopy group
$$\MU^n(X) = \lim_{\stackrel{\longrightarrow}{k}}[\Sigma^{2k-n}X_+, \MU(k)].$$

\section{Umkehr for vertically oriented maps}

Let $\E$ be  a multiplicative cohomology theory. An $\E$-orientation on a vector bundle $V\stackrel{\mu}{\to} M$ of rank $n$ is a Thom class $u = u_V\in \tilde{\E}(M^V)$ whose fibrewise restriction  $$ \tilde{\E}^n(M^V) \to \tilde{\E}^n({\rm Th}(V_m)) = \tilde{\E}^n(S^n)$$  is a generator for each $m$ in $M$. One then has a Thom isomorphism
\begin{equation}\label{T isom}
\phi = \phi_V : \E^r(M) \to \tilde{\E}^{r+n}(M^V), \hfive c\mto \mu^*(c)\cup u_V, \hthree u_V = \phi_V(1).
\end{equation}
This extends to  a canonical isomorphism
\begin{equation*}
\phi_{V-\ul l}: \E^r(M) \to \tilde{\E}^{r+n}(M^V) \cong \tilde{\E}^{r+n-l}(\Si^{-l} M^V) := \tilde{\E}^{r+n-l}(M^{V-\ul l}) 
\end{equation*}  
with the first arrow implemented as in \eqref{T isom}, and hence to defining a stable Thom isomorphism for virtual bundles via   $M^{W-W\pr} = M^{W+W^{\prime\prime}-\ul l}$ where $W\pr+W^{\prime\prime} = \ul l$ - note that $n$ may be negative - this is used here only  for the stable normal bundle.\\

A vertical $\E$-orientation on a  map $f:M\to X$  is defined to be an $\E$-orientation on its stable normal bundle $\Vv^{\rm st}_f$, or, equivalently, compatibly on each normal bundle $\Vv_{e_l\x f}$. One then has a Thom isomorphism 
$\E^r(M) \to \tilde{\E}^{r+l-q}(M^{\Vv_{e_l\x f}})$
which combines with the Pontryagin-Thom map \eqref{PT} to define the Umkehr/Gysin/integration-over-the-fibre map
\begin{equation}\label{umkehr1}
f_! : \E^r(M) \to \tilde{\E}^{r+l-q}(M^{\Vv_{e_l\x f}}) \stackrel{{\rm P.T.}}{\longrightarrow} \tilde{\E}^{r+l-q}(\Si^l X_+)  \stackrel{\cong}{\to} \E^{r-q}(X), 
\end{equation}
or,  more cleanly, by combining with the stable  Pontryagin-Thom map \eqref{stablePT} as
\begin{equation}\label{umkehr2}
f_! : \E^r(M) \to \tilde{\E}^{r-q}(M^{\Vv^{\rm st}_f}) \stackrel{{\rm P.T.}}{\longrightarrow} \E^{r-q}(X),
\end{equation}
and in case $M$ is itself  $\E$-orientable using \eqref{stablePT2} as
\begin{equation}\label{umkehr3}
f_! : \E^r(M) \to \tilde{\E}^{r-\dim M}(M^{-\Tm}) \stackrel{{\rm P.T.}}{\longrightarrow}    \tilde{\E}^{r-\dim M}(X^{-\Tx}) \to \E^{r-\dim M +\dim X}(X),
\end{equation}
where unmarked arrows in \eqref{umkehr1}, \eqref{umkehr2}, \eqref{umkehr3} are Thom isomorphisms, noting that $\Tm$ and  $\Vv^{\rm st}_f$  are $\E$-orientable imply $\Tx$ is  $\E$-orientable. The Umkehr map $f_!$ is functorial with
\begin{equation}\label{umkehr cup}
f_! (f^* (\s) \cup \om) = \s \cup f_!(\om).
\end{equation}
In the case \eqref{umkehr3}, when $M$ and $X$ are $\E$-orientable, $f_!$ may be constructed in its classical form as the Poincar\'{e}  dual to the  push-forward  $f_*$ in homology  
\begin{equation}\label{umkehrPD}
f_! : \E^r(M)  \stackrel{D}{\longrightarrow} \E_{\dim M - r}(M)  \stackrel{f_*}{\longrightarrow}   \E_{\dim M - r}(X)  \stackrel{D\ii}{\longrightarrow}  \E^{\dim X - \dim M + r}(X).
\end{equation}
The  Poincar\'{e}  duality isomorphisms $D$ are defined by cap product with the respective  $\E$ fundamental  homology classes defined by the orientations on $M$ and $X$. More generally, Poincar\'{e} duality in $\E$ holds  provided that the manifold $X$ is oriented in $\E$ and Spanier-Whitehead duality (S-duality) holds, then the diagram of isomorphisms 
\begin{equation}\label{E Poincare duality}
\begin{tikzcd}
\E^r(X)  \arrow{r}{{\rm {\small Thom}}} \arrow[swap]{dr}[below]{D}  &  \tilde{\E}^{r-\dim X}(X^{-\Tx}) \arrow{d}{S}\\
&  \ \ \ \E_{\dim X - r}(X)
\end{tikzcd}
\end{equation}
commutes. \lemref {v-orient} expresses the equivalence in singular cohomology of the Thom-Pontryagin \eqref{umkehr2} and Poincar\'{e} duality \eqref{umkehrPD} realisations of $f_!$.
\begin{comm}
 Atiyah-Poincar\'{e} duality holds on the multiplicative cohomology $\MSO^*$  in Sect.1 without $X$ being oriented, vertical orientability suffices, using just the Pontryagin-Thom map. In the presence of an orientation this  coincides with \eqref{E Poincare duality}  just as the Umkehr map \eqref{umkehr2}, defined without an orientation on $X$, then coincides with \eqref{umkehr3}. 
The isomorphism $\MSO_n(X):= \MSO_n(X,\Oo_X)\cong \MSO^{\dim X  - n}(X)$  of \cite{Atiyah 1} may be viewed as Poincar\'{e} duality for homology with twisted coefficients in a comparable way to classical singular homology Poincar\'{e}  duality for non-oriented manifolds \cite{Ranicki}. \\  
\end{comm} 
 
The fundamental property of the Umkehr for vertically oriented maps $f:M\to X$ is functoriality:
\begin{prop}
For $\phi: Y \to  X$  let $\mu: \phi^*(M) = \MxY\to Y$ be the fibre product pull-back, for which one has the commutative diagram \eqref{fibre prod diag} relabelled here as
\begin{equation}\label{pullbacksquare}
\begin{tikzcd}
W:= \MxY  \arrow{r}{\beta}  \arrow{d}{\mu} & M \arrow{d}{f} \\
Y \arrow{r}{\phi}   & X.
\end{tikzcd}
\end{equation}
Then,
\begin{equation}\label{reciprocity}
\mu_! \circ \beta^* = \phi^* \circ f_!.
\end{equation}
\end{prop}
\vthree
\begin{proof}
Frobenius reciprocity (Prop 12.9 \cite{CrJa}) applied to  \eqref{pullbacksquare} gives a commutative diagram of stable maps
\begin{equation*}
\begin{tikzcd}
Y_+  \arrow{r}{\phi_+}  \arrow{d}{\mu_+} & X_+ \arrow{d}{f_+} \\
W^{\Vv^{\rm st}_\mu}  \arrow{r}{\beta_+}   & M^{\Vv^{\rm st}_f},
\end{tikzcd}
\hten\Longrightarrow \hten
\begin{tikzcd}
 \E^{r-q} (M^{\Vv^{\rm st}_f})  \arrow{r}{f^*_+}  \arrow{d}{\beta^*_+} & \E^{r-q} (X) \arrow{d}{\phi^*} \\
\E^{r-q} ( W^{\Vv^{\rm st}_\mu})       \arrow{r}{\mu^*_+}       & \E^{r-q} (Y),
\end{tikzcd}
\end{equation*}
which concatenates with the naturality of the Thom isomorphism --- that since  $\Vv^{\rm st}_\mu = \beta^*\Vv^{\rm st}_f$ (\ppropref{induced orientations})
$$
\begin{tikzcd}
\E^r(M)  \arrow{r}{\mbox{{\tiny Thom}}}  \arrow{d}{\beta^*}  & \E^{r-q} (M^{\Vv^{\rm st}_f})  \arrow{d}{\beta^*_+}  \\
\E^r(W)       \arrow{r}{\mbox{{\tiny Thom}}}    &\E^{r-q} ( W^{\Vv^{\rm st}_\mu})     \end{tikzcd}
$$  commutes ---  
to give the commutativity of 
\begin{equation*}
\begin{tikzcd}
E^r(M)  \arrow{r}{\mbox{{\tiny Thom}}}  \arrow{d}{\beta^*}  & \E^{r-q} (M^{\Vv^{\rm st}_f})  \arrow{r}{f^*_+}  \arrow{d}{\beta^*_+} & \E^{r-q} (X) \arrow{d}{\phi^*} \\
E^r(W)       \arrow{r}{\mbox{{\tiny Thom}}}    &\E^{r-q} ( W^{\Vv^{\rm st}_\mu})       \arrow{r}{\mu^*_+}       & \E^{r-q} (Y),
\end{tikzcd}
\end{equation*}
and hence \eqref{reciprocity}.
\end{proof}

\vthree 

\section{Vertical characteristic classes}

We may apply the Umkehr/Gysin/integration-over-the-fibre for vertically oriented maps to define for each multi-index $J=(j_1, j_2, \dots, j_r)\<\N^\oo$ a generalised characteristic class 
$$\p^{\mso}_J : \MSO^*(X) \ox \Q \to \H^*(X,\Q)$$ 
by 
\begin{equation}\label{pJ}
\p^{\mso}_J(\om) = f^\om_!(\p_J(-\Vv^{\rm st}_{f^\om})),
\end{equation}
\vtwo
where $f^\om:M\to X$ is a representative for the dual of $\om = D^{\mso} [f^\om]$, with  $D^{\mso}$  the Atiyah-Poincar\'{e} duality isomorphism of Sect.1 -- this is independent of the choice of $f^\om \in[f^\om]$, see \thmref{pJf}.  $[f]$  abbreviates a homology class $[f,M,\a]\in \MSO_*(X)\ox\Q$. On the right-hand side, $\p_J(\zeta) = p_{j_1}(\zeta)\cdots p_{j_r}(\zeta)\in \H^{4|J|}(M,\Q)$ is the $J^{\rm th}$ Pontryagin class of a stable bundle $\zeta\in \KO(M)$. \\
 
Here, a characteristic class 
$\a(E)= 1+ \a_1(E) + \dots + \a_m(E)$, $\a_r(E) \in \H^r(M, R)$, 
on a semigroup of vector bundles $E$ which  on trivial bundles is equal to $1$  and has the exponential 
property $
\a(E + F)=\a(E) \cdot \a(F)$ is stable provided there are no odd degree terms $\a_{2j+1}(E)$, or without restriction on the $\a_r(E)$ if $R=\Z_2$. In these cases
 $\a$ extends to virtual bundles $E-E'$ in the corresponding $\K$-theory ring as
$\a(E-E'):=\a(E)/\a(E')$ in $\H^*(M, R),$ so
$
\a_0(E-E')=1,$  $\a_1(E-E')=\a_1(E)-\a_1(E'),$ $\a_2(E-E')=\a_2(E)-\a_1(E)\a_1(E')-\a_2(E'),$
and so on.  The total Chern class, the rational total Pontryagin class, and the total Stiefel-Whitney classes are each stable classes and defined on the corresponding stable normal bundle $\Vv^{\rm st}_f$.\\

The rational total Pontryagin class has the form
\begin{equation*}
\p(-\Vv^{\rm st}_f):=p(\Tm - f^* \Tx) = 1 + \p_1(-\Vv^{\rm st}_f) + \p_2(-\Vv^{\rm st}_f) + \cdots 
\end{equation*}
with $\p_j(-\Vv^{\rm st}_f) \in \H^{4j}(M, \mathbb{Z})$, and 
\begin{align*}
\p_J(-\Vv^{\rm st}_f):=\p_{j_1}(-\Vv^{\rm st}_f)\p_{j_2}(-\Vv^{\rm st}_f)\cdots \p_{j_r}(-\Vv^{\rm st}_f) \in \H^{4|J|}(X, \mathbb{Z}).
\end{align*}
The Umkehr map of this is the vertical characteristic class \eqref{pJ}. If $f=\pi$ is a submersion defining $[\pi]  \in \etaaSOX$ then 
\begin{equation*}
 \p^{\mso}_J(D^{\mso}[\pi]) := \pi_!(\p_J(T_\pi M)) = \int_{\mbox{$\mathsmaller{M/X}$}} \p_J(\mbox{$\mathsmaller{M/X}$})
\end{equation*}
 associated to   the Pontryagin class of the tangent bundle along the fibres $T_\pi M$. The  notation  
$\int_{\MX}$ refers to the de Rham realization of $\pi_!$ as  integration over the fibre on a Chern-Weil representative $\p_J(\mbox{$\mathsmaller{M/X}$})$ of the given classes. \\

The choice of orientation $\a$ in $[f^\a]:=[f]=[f,M,\a]$ has been suppressed in the above, but the vertical Pontryagin class is sensitive to it, even though   $\p_J(-\Vv^{\rm st}_f)$ is not. When the  orientation $\a$ on  $\Vv^{\rm st}_f$ is reversed  to $-\a$ the Thom class then changes sign and hence
$$f^{-\a}_! = - f^\a_! \, ,$$ and hence: 
\begin{lem}\label{alphasign}  
\begin{equation}\label{pJalpha}
\p^{\mso}_J(D^{\mso}[f,M,-\a])= - \,\p^{\mso}_J(D^{\mso}[f,M,\a]).
\end{equation}
\end{lem}

\vthree 

We note also the naturality  of the vertical classes. 
\begin{prop}
For a map $\phi: Y \to  X$ one has for each $\om\in\MSO^*(X)$
\begin{equation}\label{pJ pullback}
\phi^*\p^{\mso}_J(\om) = \p^{\mso}_J(\phi^*\om).
\end{equation}
\end{prop}
\begin{proof} 
Writing $\om= D^{\mso}[f]$ and with reference to the diagram of \eqref{pullbacksquare}, we have 
\begin{eqnarray*}
\phi^*\p^{\mso}_J(\om) & = & \phi^* f_!(\p_J(-\Vv^{\rm st}_{f})) \\
& \stackrel{\eqref{reciprocity}}{=}  & \mu_! (\beta^*(\p_J(-\Vv^{\rm st}_f)))  \\
& =  & \mu_! (\p_J(-\beta^*\Vv^{\rm st}_f)  \\
& = & \mu_! (\p_J(-\Vv^{\rm st}_\mu)) \\
& = & \p^{\mso}_J(\om^\mu)
\end{eqnarray*}
with  $\om^\mu = D^{\mso}[\mu]$. The fourth equality is the  stable equivalence \eqref{Vv pull back} of $\Vv^{\rm st}_\mu$ and $\beta^*(\Vv^{\rm st}_f)$. The claim, then, setting $\om^f :=\om$, is that
\begin{equation}\label{ommu=omf}
\om^\mu = \phi^* \om^f .
\end{equation}
But, from Sect.1,  $ \om^f $ is the Pontryagin-Thom classifying map
$$X_+ \stackrel{f_+}{\too} M^{\Vv^{\rm st}_f} \too \MSO_{-q},$$
which pulls-back by $\phi$ to 
$$Y_+ \stackrel{\phi_+}{\too}  X_+ \stackrel{f_+}{\too} M^{\Vv^{\rm st}_f} \too \MSO_{-q},$$
while $ \om^\mu $ is the Pontryagin-Thom classifying map
$$Y_+ \stackrel{\mu_+}{\too} M^{\Vv^{\rm st}_\mu} \too \MSO_{-q},$$
and these fit together in the diagram
\begin{equation*}
\begin{tikzcd}[row sep=scriptsize, column sep=1mm]
Y_+ \arrow{rr}{\phi_+}  \arrow{dd}{\mu_+} & & X_+  \arrow{dd}{f_+} &  \\ 
 &   &  &  \\ 
 (M\xX Y)^{\Vv^{\rm st}_\mu}  \arrow{rr}{\beta_+} \arrow{dr} & & M^{\Vv^{\rm st}_f}  \arrow{dl} &  \\ 
 & \MSO_{-q}  &  & 
\end{tikzcd}.
\end{equation*} 
The top square we know to commute from \eqref{pullbacksquare}, while the lower triangle commutes up to homotopy because 
$\Vv^{\rm st}_\mu= \beta^*(\Vv^{\rm st}_f)$ in $\K(M)$. Hence \eqref{ommu=omf} holds in $[Y_+, \MSO_{-q}]$, i.e. in  $\MSO^{-q}(Y)$. 
\end{proof}

\vtwo 

The characteristic classes \eqref{pJ} are well-defined  invariants of the vertically oriented bordism class $[f,M,\a]$: \\

\begin{theorem}\label{pJf}
If $(f,M,\a) \sim (f\pr,M\pr,\a\pr)$ (are bordant in $\MSO_*(X)$) then 
\begin{equation}\label{vert  pJ}
f_!(\p_J(-\Vv^{\rm st}_f)) = f\pr_!(\p_J(-\Vv^{\rm st}_{f\pr}))
\end{equation}
 in $\H^{*}(X, \Q)$ for each $J\in\N^\oo$. Hence $\p_J(D^{\mso} ([f])):=f_!(\p_J(-\Vv^{\rm st}_f))$ gives a well-defined   homomorphism of abelian groups
\begin{equation*}
\p^{\mso}_J: \MSO^*(X) \ox \Q\to \H^{*}(X, \Q). 
\end{equation*}
\end{theorem} 
 
\vtwo  

\begin{proof}
In singular cohomology the Poincar\'{e} duality  construction \eqref{umkehrPD} of the Umkehr map extends to vertically-oriented $f$. Precisely, Poincar\'{e} duality holds for a closed non-orientable manifold $N$  in singular cohomology  as a cap product isomorphism
\begin{equation*}
\H^r(N,\Z) \stackrel{D}{\to} \H_{\dim N - r} (N,\Z^{\sww})  
\end{equation*}
with $\H_k(N,\Z^{\sww})$ the homology of the twisted singular chain complex \mbox{$S_k(\mathcal{O}_N)\ox_{\Z[\Z_2]}\Z^-$} of the orientation double cover $\mathcal{O}_N$, where   $\Z^-$ is $\Z$ as the right $\Z[\Z_2]$-module in which the generator of $\Z_2$ acts by $-1$, reducing to $\H_k(N,\Z)$ if $N$ is orientable; see Ch.4 of \cite{Ranicki}. The point here is that $f:M\to X$ being vertically oriented means a specific isomorphism   $ \a_{\mbox{\tiny M,X}} : \mathcal{O}_M \to f^* \mathcal{O}_X$ and hence a canonical bundle map $  \mathcal{O}_M \to  \mathcal{O}_X$ and hence a map of twisted chain complexes defining a push-forward 
\begin{equation*}
 f^{\sww}_*: \H_k (M,\Z^{\sww})   \to  \H_k (X,\Z^{\sww}),
\end{equation*}
and hence the extension of   \eqref{umkehrPD} as
\begin{equation}\label{nonorientable umkehr}
f_! : \H^r(M,\Z)  \stackrel{D}{\longrightarrow} \H_{\dim M - r}(M,\Z^{\sww})    \stackrel{f^{\sww}_*}{\longrightarrow}  \H_{\dim M - r}(X,\Z^{\sww})    \stackrel{D\ii}{\longrightarrow}  \H^{r-q}(X,\Z).
\end{equation} 
 
This extends to a vertically oriented map $\s: W \to X$ on a manifold $W$ with boundary $M := \pd W\neq \emptyset$ (defining a bordism) in an analogous manner as
\begin{equation}\label{nonorientable umkehr with boundary}
\s_! : \H^r(W,M, \Z)  \stackrel{D}{\longrightarrow} \H_{\dim M - r}(W,\Z^{\sww})    \stackrel{\s^{\sww}_*}{\longrightarrow}  \H_{\dim M - r}(X,\Z^{\sww})    \stackrel{D\ii}{\longrightarrow}  \H^{r-q}(X,\Z),
\end{equation}
where $D$, here, is Leftschetz-Poincar\'{e} duality. \\

For $(f,M,\a), (f\pr,M\pr,\a\pr)$ representing elements of $\MSO_n(X)$, the difference element is represented by $(f\sqcup f\pr, M \sqcup M\pr,  \a\sqcup (-\a\pr))$ with, using \lemref{alphasign}, vertical Pontryagin class
$$(f\sqcup f\pr)_!(\p_J(-\Vv^{\rm st}_{(f\sqcup f\pr,  \a\sqcup (-\a\pr)})) = f_!(\p_J(-\Vv^{\rm st}_{(f,\a)})) - f\pr_!(\p_J(-\Vv^{\rm st}_{(f\pr,\a\pr)}))$$
\eqref{vert  pJ} is thus equivalent to $(f\sqcup f\pr)_!(\p_J(-\Vv^{\rm st}_{(f\sqcup f\pr,  \a\sqcup (-\a\pr)})) =0$ if $(f,M,\a) \sim (f\pr,M\pr,\a\pr)$. So it is enough to show 
\begin{equation}\label{vert  pJ 0}
f_!(\p_J(-\Vv^{\rm st}_f)) = 0 \hfive \ffor \ \  (f,M,\a)\sim 0 \ \ \iin \  \ \MSO_n(X),
\end{equation}
that is, for $f:M\to X$ the restriction to  $\partial W =M$ of vertically oriented $\sigma: W \to  X$, so  
\begin{equation}\label{s o i =f}
\sigma\circ i = f
\end{equation}
 where $i:M\to W$
is the inclusion map, and check $f_!(\p_J(-\Vv^{\rm st}_f))$ is invariant with respect  to fibrewise diffeomorphisms of $M\to N$ preserving the vertical orientation. \\

To see \eqref{vert  pJ 0}, there is the commutative diagram
\begin{equation}\label{diagramusefulforendofproof}
\begin{tikzcd}
 \H^k(W,\Z) \arrow{r}{i^*} \arrow{d} & \H^k(M,\Z)\arrow{r}{\delta} \arrow{d} & \H^{k+1}(W, M,\Z)\arrow{r}\arrow{d} & \H^{k+1}(W,\Z)\arrow{d}  \\
 \H_{m +1 - k}(W, M,\Z^{\sww}) \arrow{r} & \H_{m-k}(M,\Z^{\sww}) \arrow{r}{i_*} & \H_{m-k}(W,\Z^{\sww})\arrow{r} & \H_{m-k}(W, M,\Z^{\sww})\end{tikzcd}
\end{equation}
in which the vertical maps are the Poincar\'e duality maps and the horizontal sequences are exact. This is  Prop 9., Ch. 8,  of  \cite{Dold} modified to Poincar\'{e} duality for vertically oriented maps. Combining with \eqref{nonorientable umkehr} and \eqref{nonorientable umkehr with boundary}  one obtains 
\begin{equation}\label{keystone}
\begin{tikzcd}[row sep=scriptsize, column sep=scriptsize] 
& \H_{m-k}(W,\Z^{\sww}) \arrow[leftarrow]{dl}[above]{i_*}  \arrow[rr, leftarrow] \arrow[bend left=30]{dd} & & \H^{k+1}(W, M,\Z) \arrow[leftarrow]{dl}[above]{\delta} \arrow[bend left=30, dashed]{dd}{\sigma_!} \\ 
\H_{m-k}(M,\Z^{\sww})  \arrow{dr}{f_*} \arrow[rr,crossing over, leftarrow]  & & \H^k(M,\Z) \arrow[dashed]{dr}[above]{f_!} \\ 
& \H_{m-k}(X,\Z^{\sww})  \arrow[rr] & & \H^{k-q}(X,\Z) 
\end{tikzcd}
\end{equation}
(where $\H_{m-k}(W,\Z^{\sww}) \to  \H_{m-k}(X,\Z^{\sww})$ is $\sigma_*$) in which the solid arrows form a commutative array. That commutativity implies that the right-hand triangle also commutes
\begin{align}\label{nicecommutativityproperty}
f_!= \sigma_! \circ \delta.
\end{align}
Consider 
$$\p_J(-\Vv^{\rm st}_f)\in \H^{4|J|}(W,\Z).$$
The boundary $M= \pd W$ has a collar neighbourhood $\mathcal{U}_M:= \partial W \times [0, 1) \<W$ and so the restriction of $\Tw $ to $\pd W$ has the form 
$ i^* \Tw  \cong \Tm  + \ul 1.$
Hence 
\begin{equation}\label{VfVs}
i^* \Vv^{\rm st}_\s  = i^* (\s^* \Tx  - \Tw ) = i^* \s^* \Tx  - \Tm  - \ul 1 = f^* \Tx - \Tm  - \ul1 = \Vv^{\rm st}_f - \ul1.
\end{equation}
Hence, as since $\p_J$ is stable,
$\p_J(-\Vv^{\rm st}_f)   = \p_J(-\Vv^{\rm st}_f + \ul1)   = i^* \p_J(-\Vv^{\rm st}_\s).$
Thus, 
$$
f_!(\p_J(-\Vv^{\rm st}_f)) =  f_! i^* \p_J(-\Vv^{\rm st}_\s)
 \stackrel{(\ref{s o i =f})}{=}  \sigma_! \circ \delta (  i^* \p_J(-\Vv^{\rm st}_\s)) = \sigma_!  (\delta \circ i^*)(\p_J(-\Vv^{\rm st}_\s)) = 0$$
by the exactness of  (\ref{diagramusefulforendofproof}). \\

To see that $f_!(\p_J(-\Vv^{\rm st}_f)) = g_!(\p_J(-\Vv^{\rm st}_g)) $ for $f:M\to X$ and $g:N\to X$ vertically oriented maps with
$f = g\circ \psi$
and $\psi : M\to N$ a vertical-orientation preserving diffeomorphism, one in this case has $\Tm\cong \psi^* \Tn$, and hence  $\Vv^{\rm st}_f  = \psi^* \Vv^{\rm st}_g$, and hence 
\begin{align*} 
f_!(\p_J(-\Vv^{\rm st}_f))= g_!\psi_! \p_J(-\psi^* \Vv^{\rm st}_g) & =g_!(\psi_!( \psi^* \p_J(-\Vv^{\rm st}_g))) \\ 
& = g_!(\p_J(-\Vv^{\rm st}_g)\psi_! (1)) \\ & = g_!(\p_J(- \Vv^{\rm st}_g))
\end{align*}
since for a diffeomorphism $\psi_! (1)=\pm 1$ with the positive sign taken if preserving the vertical orientation. \\

This continues to hold even for $\psi$ a homeomorphism, preserving vertical orientation. For, in this case Novikov has shown that the rational Pontryagin classes are  topological invariants --- one has 
$\psi^* \p_j(N) = \p_j(M)$ over $\Q$. But $\p(N) = \p(-\Vv^{\rm st}_g) g^* \p(X),$ with $p$ the total Pontryagin class, and a similar formula for $p(M)$. Equating using the Whitney sum formulae, one iteratively infers   $\phi^* \p_j(-\Vv^{\rm st}_g)= \p_j(-\Vv^{\rm st}_f)$ for all $j$, and the result follows as before. 
\end{proof} 

This gives one direction of Theorem 1. \\

Corresponding statements are immediate for the vertical Chern and Stiefel-Whitney classes on, respectively, $\MU^*(X)$ and $\MO^*(X)$: one has 
$$\c^{\muu}_J : \MU^*(X) \to \H^r(X,\Z)$$
defined by 
\begin{equation*}
\c^{\muu}_J(D^{\muu} [f]) = f_!(\c_J(-\Vv^{\rm st}_f)),
\end{equation*}
\vtwo
where $\c_J(-\Vv^{\rm st}_f) := \c_J(\zeta) = \c_{j_1}(\zeta)\cdots \c_{j_r}(\zeta)$ is the  $J^{\rm th}$ Chern class of  a stably equivalent complex virtual bundle $\zeta\in \K(M)$ defining the stable vertical complex structure on $-\Vv^{\rm st}_f$ --- this is independent of the choice of $\zeta$. Since a stable complex structure on  $\Vv^{\rm st}_f$  defines an orientation on it, the above proof also implies 
\begin{equation*}
\mu =0 \ {\rm in}  \ \MU^*(X) \  \Rightarrow  \c^{\muu}_J (\mu) =0 \  \,\forall J\<\N^\oo.
\end{equation*}

Likewise, 
vertical Stiefel-Whitney classes
$$\sw^{\mo}_J : \MO^*(X) \to \H^r(X,\Z_2)$$ are
defined by 
\begin{equation*}
\sw^{\mo}_J(D^{\mo} [f]) = f_!(\sw_J(-\Vv^{\rm st}_f)),
\end{equation*}
which collectively vanish when $f$ is an unoriented boundary map. 
\vtwo
In the latter case matters are simplified because any manifold is  $\H\Z_2$-oriented, so the proof of \thmref{pJf} goes through unchanged with $\Z_2$ replacing $\Z$ and the twisted coefficients $\Z^{\sww}$. \\

Indeed, for any generalised cohomology theory $\E$ which is multiplicative with Poincar\'{e} duality  and for which there is a stable $\E$-characteristic class  $\b: \K(N) \to \E(N)$, defining $\b_J(D^\E [f]) := f_!(\b_J(-\Vv^{\rm st}_f))$, if \eqref{diagramusefulforendofproof} commutes, then $\mu =0$ in $\E^*(X)$ $ \Rightarrow$  $\b_J (\mu) =0$ for all $J$, as in the proof of \thmref{pJf}. \\

\begin{comm}
A  more generic proof (avoiding twisted coefficients) uses the construction \eqref{umkehr2} of $f_!$, requiring only that $f$ be vertically oriented; for this \eqref{keystone} is replaced by a diagram
\begin{equation*}
\begin{tikzcd}[row sep=scriptsize, column sep=scriptsize]
E^k(M) \arrow[bend left=30]{rr}{\d} \arrow{dddr}[below]{f_! \ \ }  \arrow{dd}{\phi^M} & & E^{k+1}(W,M)  \arrow{dddl}[below]{\ \ \sigma_!} \arrow{dd}{\phi^W} &  \\ 
 &   &  &  \\ 
 E^{k-q}(M^{\Vv_f^{st}}) \arrow[rr, bend left=30, leftarrow, crossing over]{L^*} \arrow{dr}[below]{P.T. \ \ } & & E^{k-q}(W^{\Vv^{st}_{\sigma}}) \arrow{dl}[below]{ \ \ P.T.}  &  \\ 
 & E^{k-q}(X)  &  & 
\end{tikzcd}
\end{equation*}  
where the unmarked arrow is induced by the map  $M^{\Vv_f^{st}}\to W^{\Vv^{st}_{\sigma}}$ defined via \eqref{VfVs}. The rest of the proof proceeds unchanged when commutativity holds. 
\end{comm}

\section{The converse: $\p^{\mso}_J (\mu) =0$ $\forall J$ $\Rightarrow $  $\mu =0$ in $\MSO^*(X)$ }

\vthree 

For orientable  $X$ the result is inferred from classical results of Conner and Floyd, as follows. In this case $[f]\in \MSO_*(X)$ is represented by $f:M\to X$  with $M$ $\H\Z$-orientable. Thus, one has generating fundamental classes $[M]\in \H_n(M,\Z)$ and $[X]\in \H_n(X,\Z)$. \\

A rational homology class $\b\in \H_*(X,\Q)$ is uniquely determined by the Kronecker pairing map  $$ \langle \cdot\, ,\b \rangle  : \H^*(X,\Q) \to \Q.$$ Define $q^{\mso}_J([f])\in \H_*(X,\Q)$   by setting  -- for arbitrary $\om\in \H^*(X,\Q)$ since $M$ is orientable -- 
$$ \langle \om, q^{\mso}_J( [f]) \rangle  \ = \ \langle f^*\om\cup \p_J(-\Vv^{\rm st}_f), [M] \rangle .$$
Using \eqref{umkehr cup},  the right-hand side is
$$ \langle \om\cup f_!(\p_J(-\Vv^{\rm st}_f)), [X] \rangle  \ =  \ \langle \om\cup \p^{\mso}_J(D^{\mso} [f])), [X] \rangle  \ =  \ \langle \om, D (\p^{\mso}_J(D^{\mso} [f])) \rangle $$
where the second equality uses the property $\langle \om\cup\nu, [X] \rangle  \ = \ \langle \om ,D\nu \rangle $ with $D$ Poincar\'{e} duality on $\H^*(X,\Q)$, and hence:
 
\begin{lem}\label{qf}
$q^{\mso}_J([f])\in \H_*(X,\Q)$ is Poincar\'{e} dual to $\p^{\mso}_J(D^{\mso} [f])\in \H^*(X,\Q)$.  
\end{lem}

The vanishing of each of the classes $\p^{\mso}_J(D^{\mso} [f])$ therefore implies the same for the homology classes $q^{\mso}_J([f])$ and hence $ \langle f^*\om\cup \p_J(-\Vv^{\rm st}_f), [M] \rangle  = 0$ for each $\om\in \H^*(X,\Q)$ and $J\<\N^\oo$ (specifically, for  $4|J| + |\om| = \dim M$), in this case.
Taking 
$$\om = c\cup p_{J\pr}(X)$$
we infer that 
$$ \langle f^*c \cup  \p_{J\pr} (f^*\Tx)\cup \p_J(-\Vv^{\rm st}_f), [M] \rangle  = 0$$
for each $c\in \H^*(X,\Q)$ and $J, J\pr\<\N^\oo$. But $\p_I(M) = \sum_{J\sqcup J\pr = I} \p_{J\pr} (f^*\Tx)\cup \p_J(-\Vv^{\rm st}_f)$
and hence 
\begin{equation}\label{CFnumbers}
 \langle f^*c \cup  \p_I(M), [M] \rangle  = 0
\end{equation}
for each $c\in \H^*(X,\Q)$ and $I\<\N^\oo$.  The rational numbers on the left-hand side of \eqref{CFnumbers} are the Conner-Floyd characteristic numbers for bordism classes of maps of oriented manifolds, whose collective vanishing was shown in \cite{Conner-Floyd1} to imply (in fact, to be equivalent to) $[f]=0$:\\

Thus, Theorem  1 \ holds \ when \ $X$ \ is \ $\H\Z$-orientable.\\

Replacing $\p_J$ by $\sw_J$ and $\Q$ by $\Z_2$, the same argument proves the unoriented bordism version of Theorem 1 (the  $\Rightarrow$ direction having been shown here  in \S3):
\begin{equation}\label{eqn swJ=0}
\a = \a\pr\ {\rm in}\ \MO^*(X) \ \ \Longleftrightarrow \ \ \sw^{\mo}_J(\a)=\sw^{\mo}_J(\a\pr) \ \ {\rm in}\ \H^*(X,\Z_2) \ \ \forall \ J \subset \mathbb{N}^\infty.
\end{equation}
$\MO^*(X)$ coincides with the unoriented bordism theory of \cite{Conner-Floyd1} and \eqref{eqn swJ=0} is equivalent to their result. A different proof of $\Leftarrow$ is given here below, and for $\etaaX$ can alternatively be shown using a stable fibrewise Thom map.\\

However, our purpose is to prove Theorem 1 for $\MSO^*(X)\ox\Q$, for which $X$ need not be orientable, and also to prove the corresponding vertical Chern class characterisation of  complex bordism cohomology  $\MU^*(X)$ --- that 
for $\mu\in\MU^*(X)$
\begin{equation}\label{eqn cJ=0}
\c^{\muu}_J (\mu) =0 \  \,\forall J\<\N^\oo \ \Rightarrow  \mu =0 \ {\rm in}  \ \MU^*(X),
\end{equation}
the converse having been shown in \S3.  We proceed via a vertical Riemann-Roch theorem on $\MU^*(X)$. This adapts ideas of \cite{Dyer} and \cite{Buchstaber}. \\

Let $\hh^* = \{\hh^k\}$ and $\kk^* = \{\kk^k\}$ be multiplicative cohomology theories. Let $\tau:\hh^*\to \kk^*$ be a natural transformation such that  for each space $N$, $\tau:\hh^*(N)\to \kk^*(N)$ is a multiplicative (ring) homomorphism, and                                                     
such that if $\a\in\hh^1(S^1, {\rm pt})$ and $\b\in\kk^1(S^1, {\rm pt})$ are suspensions of the units in $\hh^0(S^1, {\rm pt})$ and $\kk^0(S^1, {\rm pt})$, then $\tau(\a) = \b$.  $\tau$ is then said to be a multiplicative transformation \cite{Dyer}. \\

Let  $\xi$ be a stable vector bundle over a manifold $M$, defining a class in the corresponding \K ring $ (M)$ on $M$ of stable bundles with a specified $G$-structure. A multiplicative transformation $\tau:\hh^*\to \kk^*$ may be used to associate to $\xi$ a generalised Todd class 
\begin{equation}\label{gen Todd}
T_\t(\xi) \in \kk^*(M).
\end{equation}
For this, suppose $\xi$ is $\hh$-oriented and $\kk$-oriented, so one has Thom isomorphisms
$$\phi_\hh = \phi_{\hh, \xi} : \hh^r(M) \to \tilde{\hh}^{r + \rk(\xi)}(M^\xi)$$
 and 
$$\phi_\kk = \phi_{\kk, \xi} : \kk^r(M) \to \tilde{\kk}^{r + \rk(\xi)}(M^\xi).$$
Consider the composition
$$\hh^*(M)\stackrel{\phi_\hh}{\to} \tilde{\hh}^*(M^\xi)\stackrel{\t}{\to} \tilde{\kk}^*(M^\xi) \stackrel{\phi_\kk\ii}{\to}\kk^*(M)$$
and assume that for each $n\in\N$
\begin{equation}\label{stable tau}
\t(\phi_{\hh, \ul n}) = \phi_{\kk, \ul n}.
\end{equation}
Then, define \eqref{gen Todd} by 
\begin{equation*}
T_\t(\xi) = \phi_{\kk, \xi}\ii\, \t \,\phi_{\hh, \xi} (1) = \phi_{\kk, \xi}\ii\, \t (u_{\hh, \xi}(\xi))
\end{equation*}
with $u_\hh (\xi) =\phi_{\hh, \xi}(1) \in  \tilde{\hh}^*(M^\xi)$  the Thom class. For $m\in\hh^*(M)$ define 
\begin{equation*}
T_\t(\xi)(m)  = \phi_{\kk, \xi}\ii\, \t \,\phi_{\hh, \xi} (m)  \in \ \kk^*(M).
\end{equation*}

\begin{lem}\label{stable t-Todd}
The class $\xi\mapsto T_\t(\xi)$ is stable, i.e. $T_\t$ pushes-down to a group homomorphism $T_\t: \K(M)\to \kk^*(M)$ (from the appropriate $\K$-theory). 
\end{lem}
\begin{proof} The Thom class has the functoriality property 
\begin{equation}\label{phihh hom}
\phi_\hh (\xi + \eta) = p_1^*\phi_\hh (\xi) \cup p_2^*\phi_\hh (\eta)
\end{equation}
where $p_1, p_2$ are the respective projection maps onto $\xi$ and $\eta$, and similarly for $\phi_\kk$. Hence, since
$\t$ is multiplicative,
\begin{equation}\label{T hom}
T_\t(\xi + \eta) = T_\t(\xi)T_\t(\eta).
\end{equation}
From \eqref{stable tau} \eqref{phihh hom} and \eqref{T hom} and that $T_\t(\ul n) = 1_\hh$, we have $T_\t(\xi+\ul n) = T_\t(\xi)$ and the result follows.
\end{proof}

Thus we have a canonical map $ \MSO_*(X) \to \kk^*(M)$ given by 
$[f] \mapsto T_\t(\Vv^{\rm st}_f),$
or, more concretely, relative to an embedding $e_l: M\to \R^l$ with oriented normal bundle $\Vv_{e_l\x f}$, by  
$[f] \mapsto T_\t(\Vv_{e_l\x f}),$
and similarly for  $MU_*(X)$ with a (stable) complex structure class on $\Vv_{e_l\x f} + \ul m$ for $m$ sufficiently large, the map is $[f] \mapsto T_\t(\Vv_{e_l\x f} + \ul m)$.\\
\begin{lem}\label{T(m)} \cite{Dyer}
One has  for a virtual bundle $\xi$
\begin{equation*}
T_\t(\xi)(m) = T_\t(\xi)\cup \t(m).
\end{equation*}
\end{lem}
\begin{proof} It is enough, in view of \lemref{stable t-Todd}, to prove this for a vector bundle  $\xi$. 
Let $\pi:\xi\to M$ be the bundle projection map. Then, with $\phi_\kk = \phi_{\kk,\xi}, \phi_\hh = \phi_{\hh,\xi}$,
\begin{eqnarray*}
\phi_\kk\ii   \t    \phi_\hh(m) &=& \phi_\kk\ii  \t (\phi_\hh(1) \cup \pi^* m) \\
					&=& \phi_\kk\ii  ( \t (\phi_\hh(1)) \cup \t(\pi^* m)) \\
					&=& \phi_\kk\ii  ( \t (\phi_\hh(1)) \cup \pi^*\t(m)) \\
					&=& \pi_! ( \t (\phi_\hh(1)) \cup \pi^*\t(m)) \\
					&=& \pi_! \l( \t (\phi_\hh(1))\r) \cup \t(m) \\
					&=& \phi_\kk\ii \t (\phi_\hh(1)) \cup \t(m)
\end{eqnarray*}
\end{proof}

The following is a tweak of Dyer's Riemann-Roch theorem for oriented maps in generalised cohomologies \cite{Dyer} to the case of vertically oriented maps: 

\begin{theorem}\label{Thm RR}
Let $f:M\to X$ be a continuous map of manifolds which is both $\hh$ and $\kk$ vertically oriented, so that the Umkehr maps $f^\hh_!$ and $f^\kk_!$ are defined and the vertical Todd class  $T_\t(\Vv^{\rm st}_f)$ is defined. Then for $m\in\hh^*(M)$ one has in $\kk^*(M)$
\begin{equation}\label{vertical RR}
\t(f^\hh_!(m)) = f^\kk_! (T_\t(\Vv^{\rm st}_f)\cup \t (m)).
\end{equation}
\end{theorem}
\begin{proof} Since $\Vv^{\rm st}_f$ is $\hh$-oriented and $\kk$-oriented we have Thom isomorphisms
$$\phi_{\hh, \Vv^{\rm st}_f}  : \hh^r(M) \to \tilde{\hh}^{r -q}(M^{\Vv^{\rm st}_f}),\ \ \aand \ \ \phi_{\kk, \Vv^{\rm st}_f}   : \kk^r(M) \to \tilde{\kk}^{r -q}(M^{\Vv^{\rm st}_f}).$$
Combining with the Pontryagin-Thom map $\tilde f : X_+\to M^{\Vv^{\rm st}_f}$, inducing
$$ \tilde f^*_\hh : \tilde{\hh}^*(M^{\Vv^{\rm st}_f}) \to \hh^*(X_+) \ \ \aand \ \ \tilde f^*_\kk : \tilde{\kk}^*(M^{\Vv^{\rm st}_f}) \to \kk^*(X_+),$$
gives the Umkehr/Gysin maps 
 \begin{equation}\label{umkehr RR}
 f^\hh_! = \tilde f^*_\hh \circ \phi_\hh  : \hh^r(M) \to \hh^{r-q}(X) \ \ \ \aand \ \ \ f^\kk_! = \tilde f^*_\kk \circ \phi_\kk  : \kk^r(M) \to \kk^{r-q}(X),
\end{equation}
and we have 
\begin{eqnarray*}
\t(f^\hh_!(m))  & =  & \t \tilde f^*_\hh  \phi_\hh(m) \\
			& =  & \tilde f^*_\kk \t    \phi_\hh(m) \\
			& =  & \tilde f^*_\kk \phi_\kk \l( \phi_\kk\ii   \t    \phi_\hh(m)\r)\\
			& =  & f^\kk_! \l( \phi_\kk\ii   \t    \phi_\hh(m)\r)\\
			& =  & f^\kk_! \l( T_\t(\Vv^{\rm st}_f)(m)\r)
\end{eqnarray*}
and so using \lemref{T(m)} and stability of $T_\t$ we reach \eqref{vertical RR}.
\end{proof}

In the case, making a stronger assumption, that $M$ and $X$ are $\hh$, $\kk$ orientable, one also has the (then) equivalent formulation of
\eqref{vertical RR} of \cite{Dyer}. Namely, $M$ and $X$ orientability gives Thom isomorphisms $\phi_{\hh, -\Tm}  : \hh^r(M) \to \tilde{\hh}^{r -q}(M^{-\Tm})$ and $\phi_{\kk, -\Tm}$, and likewise for the stable normal bundle $-\Tx$, yielding the construction \eqref{umkehr3}  of the Umkehr maps $f^\hh_!$ and $f^\kk_!$ and generalised Todd classes of the stable normal bundles. Dyer's Riemann-Roch formula is then 
\begin{equation}\label{horizontal RR}
\t(f^\hh_!(m))\cup T_\t(-\Tx) = f^\kk_! (T_\t(-\Tm)(m))
\end{equation}
which in view of  \eqref{umkehr cup} and \eqref{T hom} is the same as \eqref{vertical RR}. On the other hand, \eqref{vertical RR} holds for vertically oriented $f$ when $X$ is not orientable and \eqref{horizontal RR} is not applicable. \\

The total Novikov operation  \cite{Adams, Quillen}
\begin{equation}\label{Novikov S}
\Sbb = \sum_J \Sbb_J: \MU^*(X)\to \MU^*(X)
\end{equation}
summed over multi-indices $J$, where $\Sbb_J: \MU^k(X)\to \MU^{k+ 2|J|}(X)$, has the properties that $\Sbb_0$ is the identity, that $f^* \Sbb_J = \Sbb_J f^*$, and $\Sbb$ is multiplicative, and is of some tangential interest insofar as for a complex vector bundle $\xi\to M$ 
$$T_s(\xi) = \sum_J \cf^{\muu}_J(\xi)$$
with $\cf^{\muu}_J(\xi) = \cf^{\muu}_{j_1}(\xi)\cdots \cf^{\muu}_{j_r}(\xi)$, where $ \cf^{\muu}_j(\xi)$ is the $j^{{\rm th}}$ Conner-Floyd Chern class of $\xi$ \cite{Adams, Kochman}. 
So \eqref{vertical RR} says: 
\begin{prop}
For a vertical stably complex map $f:M\to X$ one has in $\MU^*(X)$
$$\Sbb(f^{\muu}_!(\cf^{\,\,\muu}_J(\xi))) = f^{\muu}_! (\cf^{\,\,\muu}_J(\Vv^{\rm st}_f)\cup \Sbb(\cf^{\,\,\muu}_J(\xi))).$$
\end{prop}
\vthree

Here, any complex orientable cohomology theory has Conner-Floyd Chern classes associated to each complex bundle and these behave naturally with respect to Whitney sum and pull-back. The classes $\cf^{\,\hh}_k\in \hh^{2k}(\BU)$ of the universal bundle provide a canonical set of  generators
for the polynomial ring $\hh^*(\BU) = \hh^*[\cf^{\,\hh}_1, \cf^{\,\hh}_2, \ldots]$  \cite{Conner-Floyd2} , \cite{Adams}, \cite{Kochman}. This mirrors the classical Chern classes -- to which they reduce when $\hh= \H\Z^*$ -- however, the ring structure for general $\hh$ is  more subtle, depending on a formal group law.\\

The Novikov operation \eqref{Novikov S} is a homolog in $\MU^*$ of the Steenrod operation on $\H\Z_2$ 
\begin{equation*}
Sq = \sum_{k=0} Sq_k : \H^*(N, \Z_2) \to  \H^*(N, \Z_2),
\end{equation*}
where $Sq_k : \H^m(N, \Z_2) \to  \H^{m+ k}(N, \Z_2)$ is the $k^{{\rm th}}$ Steenrod square, for which one has the classical Stiefel-Whitney class formula
$$T_{Sq}(\zeta) = \sum_k \sw_k(\zeta) =: \sw(\zeta),$$
and \eqref{vertical RR} says  for a real bundle $\zeta\to M$
\begin{equation*}
Sq(f^{\Hsss}_!(\sw(\zeta))) = f^{\Hsss}_!(\sw(\Vv^{\rm st}_f)\cup Sq(\sw(\zeta))) \hfive \iin \ \H^*(X,\Z_2),
\end{equation*}
similarly to \cite{AtHi} and \cite{Dyer}.\\

The  multiplicative transformation we have specific need of here is the rational Chern-Dold character which defines a canonical map from a multiplicative cohomology $\hh^*$ 
\begin{equation*}
\ch^\hh : \hh^*(N)\ox\Q \too \H^*(N, \L^*_\hh \ox\Q) := \sum_{l\geq 0} \H^l(N, \L^{l-1}_\hh \ox\Q)
\end{equation*}
to singular cohomology with coefficients in the rational cobordism ring  $\L^*_\hh \ox\Q$. On finite complexes $\ch^\hh$ is a ring isomorphism. For, \cite{Dyer} Cor.4 proves that the rationalisation of any multiplicative (co)homology theory is singular cohomology, that there is an isomorphism 
$ \hh^*(N)\ox\Q \cong \H^*(N, \L_\hh^* \ox\Q)$
and by \cite{Dyer} Th.2 there is a unique such isomorphism which on $N = {\rm pt}$ coincides with the identity map 
$\L_\hh^* \ox\Q \to \L_\hh^* \ox\Q$ -- this is $\ch^\hh$. \\ 

We now apply \thmref{Thm RR} to the Chern-Dold character on vertical bordism cohomology. \\

First, note how it applies to the classical case 
$\hh = \K,  \kk = \H^* $
 for which the Chern-Dold character reduces to the usual Chern character 
$\ch^{\uu} : \K(N)\ox \Q \too \H^*(N, \Q),$
and for $\xi$ a stable complex bundle 
\begin{equation}\label{classical Todd xi}
T(-\xi) = \phi_{\Hss,-\xi}\ii  \ch^{\uu}   \phi_{\Kss,-\xi}(1) = \Td(\xi)
\end{equation}
with $\Td$ the classical Todd class, which augments for $m=E\in \K(M)$ to 
\begin{equation*} 
\phi_{\Hss,-\xi}\ii   \ch^{\uu}   \phi_{\Kss,-\xi}(E) = \Td(\xi)\,\ch^{\uu}(E),
\end{equation*}
(see for example \S12 of \cite{LaMi}).
 By \eqref{vertical RR}, the Riemann-Roch theorem therefore extends to vertically oriented $f:M\to X$ as
\begin{equation}\label{classical vertical RR}
 \ch^{\uu} (f^{\Kss}_!(E)) = f^{\Hss}_! ( \Td(-\Vv^{\rm st}_f)\,\ch^{\uu}(E)).
\end{equation}
 $\K$-orientability means $\Vv^{\rm st}_f$ is  spin-c and in this case
$\Td(\Vv^{\rm st}_f) = e^{c_1(\Vv^{\rm st}_f) / 2} \Ahat(\Vv^{\rm st}_f).$ 
If  $M$ and $X$ are assumed to be $\K$-orientable almost complex manifolds  (and $\H$-orientable),  then $\Td(-\Vv^{\rm st}_f)) = \Td(M) f^*(\Td(X))\ii $
and matters reduce to a result of \cite{AtHi} (\S3 eq.$(i)\pr$) and \cite{Dyer},
and \eqref{classical vertical RR} takes its customary form  
\begin{equation*}
 \ch^{\uu} (f^{\Kss}_!(E)) \Td(X) = f^{\Hss}_! ( \ch^{\uu}(E)\Td(M)).
\end{equation*}
Similarly, with $\hh = \KO,  \kk = \H^* $ and $\t = \ch^{\so}$ the Pontryagin character 
one has
\begin{equation*}
 \ch^{\so} (f^{\KOss}_!(E)) = f^{\Hss}_! ( \Ahat(-\Vv^{\rm st}_f))\, \ch^{\so}(E))
\end{equation*}
\vskip 1mm
reverting when  $M$ and $X$  are spin manifolds ($\KO$-orientable) to 
$$ \ch^{\so} (f^{\KOss}_!(E)) \Ahat(X) = f^{\Hss}_! (\Ahat(M) \, \ch^{\so}(E)).$$

\vtwo
Now contemplate the case
$$\hh = \MU^*\ox\Q, \hfive \kk = \H^*( \  \ , \L^*_{\muu} \ox\Q) , $$
and the Chern-Dold character 
$$\t:=\ch^{\muu} : \MU^*(N)\ox\Q \too \H^*(N, \L^*_{\muu} \ox\Q).$$
A rank $n$ complex vector bundle $\xi$ is  both $\MU$ and $\H$ oriented. Let 
\begin{equation}\label{MU Thom class}
u_\xi = u_{\muu, \xi}\in\MU^{2n}(M)
\end{equation}
 be the Thom class  corresponding to the homotopy class of maps $M^\xi\to \MU_n$ 
defined by the classifying class $M\to \BU_n$ of $\xi$. Let 
$\phi_\xi  = \phi_{\muu, \xi}  : \MU^r(M) \to \tilde{\MU}^{r+2n}(M^\xi)$ be the resulting Thom isomorphism,  $\phi_{\muu, \xi} (1)= u_\xi$.
Let $\phi_{\Hss, \xi}(1)$ be the Thom class in $\H^*(M, \L^*_{\muu} \ox\Q)$ determined by the Thom class $\mu_\Z(\phi_{\muu, \xi}(1)) \in \H^*(M, \Z)$, where  $\mu_\Z : \MU^*(M)\to \H^*(M, \Z)$ is the Thom/Steenrod homomorphism in cohomology, and let $\phi_{\Hss, \xi}$ be the resulting Thom isomorphism. Then, mirroring  \eqref{classical Todd xi}, there is the Buchstaber-Todd class \cite{Buchstaber} 
\begin{equation*}
\Td^{\muu} (\xi): = T(-\xi) = \phi_{\Hss, -\xi} \,  \ch^{\muu}\,   \phi_{\muu, -\xi} (1) \ \in \H^*( M , \L^*_{\muu} \ox\Q) . 
\end{equation*}
Applying \eqref{vertical RR} we infer that
\begin{equation*}
\ch^{\muu} (\underbrace{f^{\muu}_!(1)}_{\in\, \MU^{-q}(X)})  = f^{\Hss}_! (\Td^{\muu} (\Vv^{\rm st}_f)) = f^{\Hss}_! (\phi_{\Hss, \Vv^{\rm st}_f} \,  \ch^{\muu}\,   \phi_{\muu, \Vv^{\rm st}_f} (1)).
\end{equation*}
To evaluate these terms, we have  from \eqref{umkehr RR} that for $[f:M\to X] \in \MU_{\dim M}(X)$ 
$$f^{\muu}_!(1) =  \tilde f^*_{\muu} \circ \phi_{\muu, \Vv^{\rm st}_f}(1),\hten 1\in\MU^0(M).$$
By the definition of  \eqref{MU Thom class}, the Thom class $\phi_{\muu, \Vv^{\rm st}_f}(1) = u_{\Vv^{\rm st}_f}$ associated to $[f]$ is the element of $\MU$ cohomology specified by the classifying map $M^{\Vv^{\rm st}_f}  \too \MU_{-q}$,
while $\tilde f^*_{\muu}$ is induced by the Pontryagin-Thom map $X_+\to M^{\Vv^{\rm st}_f} $. Their composition is thus precisely the vertical Atiyah-Poincar\'{e} dual element $D^{\muu}[f] \in \MU^{-q}(X)$  computed in Sect.1. That is, 
\begin{prop}\label{fMU1}
\begin{equation*}
f^{\muu}_!(1) = D^{\muu}[f].
\end{equation*}
\end{prop}
\vtwo 
If one additionally assumes that $M$ and $X$ are individually stable complex manifolds then the Poincar\'{e} duality construction \eqref{umkehrPD} of $f^{\muu}_!$  can be used to rederive
$$f^{\muu}_!(1) = D^{\muu} \circ f^{\muu}_* \circ (D^{\muu})\ii (1_{\muu^*}) = D^{\muu} \circ f^{\muu}_*([1]) = D^{\muu}([f]),$$
where $[1]$ is the bordism class of the identity map $1:M\to M$. \\

By \thmref{vertical RR} we so far have:
\begin{prop}
\begin{equation}\label{MU vert RR 2}
\ch^{\muu} (D^{\muu}[f])  = f^{\Hss}_! (\Td^{\,\muu} (-\Vv^{\rm st}_f)).
\end{equation}
\end{prop}
\vfour

This extends Th. 1.5 of \cite{Buchstaber} to vertically oriented $f:M\to X$.\\

By stability (\lemref{stable t-Todd}) 
$\Td^{\mso} (\Vv^{\rm st}_f)  = \Td^{\muu} (\Vv_{e_l\x f} + \ul m)$
for  the normal bundle $ \Vv_{e_l\x f}$ with $l>>0$ and $m$ such that $\Vv: =\Vv_{e_l\x f} + \ul m$ has a complex structure defining the stable complex class of $\Vv^{\rm st}_f$.  Buh\v staber (loc. cit.) proves that for any complex vector bundle $\xi$
\begin{align*}
\Td^{\muu}(\xi)= \exp \left\lbrace\sum_{i\geq 1} \frac{[N^{2i}]}{d_i} \ch^{\uu}_i(\xi)\right\rbrace
\end{align*}
for unique up to bordism stably complex manifolds $N^{2i}$, and certain specific integers $d_i$. From this we infer an expansion 
\begin{equation}\label{Todd MU cf}
\Td^{\,\muu} (-\Vv^{\rm st}_f) =   \sum_J b_J\,[M_J]\,\c_J(-\Vv^{\rm st}_f)\ \in \H^*( M , \L^*_{\muu} \ox\Q)
\end{equation} 
for  some closed manifolds $M_J$ and $b_J\in\Q$ --  \eqref{Todd MU cf} can alternatively be deduced from the identity 
\begin{equation}\label{Chern-Dold MU}
\ch^{\muu}(\cf^{\muu}_1(\xi)) = \ch_1^{\uu}(\xi) + \sum_{n\geq 1} [M^{2n}] \, \ch_{n+1}^{\uu}(\xi)
\end{equation}
of  \cite{Buchstaber} (Cor. 2.4).  From \eqref{MU vert RR 2} 
\begin{equation}\label{MU vert RR 3}
\ch^{\muu} (D^{\muu}[f])  =  \sum_J b_J\,[M_J]\,f^{\Hss}_! \l(\c_J(-\Vv^{\rm st}_f)\r)  =   \sum_J b_J\,[M_J]\,\c^{\muu}_J(D^{\muu} [f]).
\end{equation}
Hence 
$$\c^{\muu}_J(D^{\muu} [f]) =0  \ \ \forall \ J \<\N^\oo \  \Too  \  \ch^{\muu} (D^{\muu}[f]) =0$$
and hence
$$\mu := D^{\muu} [f]=0 \ \iin  \ \MU^*(X)\ox\Q$$
since  $\ch^{\muu}$ is a $\Q$  ring isomorphism, or equivalently $[f] =0$ in  $\MU_*(X)\ox\Q.$ Since any cohomology class in $\MU^*(X)$ can be so written, this completes the proof of \eqref{eqn cJ=0}.

\vthree 

Let us sketch the proof for $\MO^*(X)$ and $\MSO^*(X)\ox\Q$.  Note, first, that $\MO$ is a real oriented cohomology theory, every real vector bundle $\xi \to X$ is $\MO$-oriented with Thom class $u_\xi \in \tilde{\MO}^n(X^\xi)$. Equivalently, there is a   $w_\mo\in\tilde{\MO}^1(\R P^\oo)$ mapping to $1$ in 
$\tilde{\MO}^1(\R P^\oo) \to \tilde{\MO}^1(\R P^1) \to \tilde{\MO}^1(S^1 ) \to \tilde{\MO}^0(*)$, defining  the   first universal Conner-Floyd Stiefel-Whitney class $\swf^\mo_1(L^\oo) := w_\mo \in \MO^1(\R P^\oo)$  of the tautological  line bundle $L^\oo\to \R P^\oo$. The AHSS shows that $\MO^*(\R P^\oo) = \MO^*[w].$
Pull-back by the classifying map and the Grothendieck construction yield a full spread of Conner-Floyd Stiefel-Whitney classes $\swf^\mo_k(\xi)\in \MO^k(X).$
Naturality holds for Whitney sum and pull-back, while the ring structure depends on a logarithm of the formal group law $$l_\mo(T) = \sum_{i\geq 0} s_i T^i $$ with $s_i \in \MO^*(pt)$ such that 
$\swf^\mo_1(l_\mo(\zeta_1 \otimes \zeta_2)) := \swf^\mo_1(\zeta_1) + \swf^\mo_1(\zeta_2)$
for real line bundles $\zeta_1, \zeta_2$ \cite{Quillen}.  The $s_i$ are readily determined. 

\begin{theorem} \label{MO theorem}
\begin{equation*}
\a = \a\pr\ {\rm in}\ \MO^*(X) \ \ \Longleftrightarrow \ \ \sw^{\,\mo}_J(\a)=\sw^{\,\mo}_J(\a\pr) \ \ {\rm in}\ \H^*(X,\Z_2) \ \ \forall \ J \subset \mathbb{N}^\infty.
\end{equation*}
\end{theorem}
\begin{proof}  From  \cite{Quillen}  there is a canonical isomorphism 
\begin{equation}\label{Chern-Dold MO}
\MO^*(X) \too \H^*(X,  \Z_2) \ox \L^*_{\mo}
\end{equation}
 which reduces to the identity map on $X=pt$ and so  by Dold's theorem \cite{Dold} must coincide with the Chern-Dold character $\ch^{\mo}$. Concretely, the inverse to \eqref{Chern-Dold MO} is identified by Quillen to be the $\L^*_{\mo}$ extension of  the map $\H^*(X,  \Z_2)  \to  \MO^*(X) $ which sends an element of $\H^*(X,  \Z_2) $, identified with a real line bundle $L$, to $l_{\mo}(\sw^\mo_1(L))$.  Let  
 $\Td^\mo(-\Vv^{\rm st}_f) \in \H^*(X,  \Z_2) \ox \L^*_{\mo}$ be the associated $\MO$ Todd class. The identity $f^{\mo}_!(1) = D^{\mo}[f]$ holds just as in \eqref{fMU1} (via the forgetful functor) and hence from \eqref{vertical RR} we have that 
 $\ch^\mo (D^\mo[f])  = f^{\Hss}_! (\Td^\mo (-\Vv^{\rm st}_f)).$
 Entirely similar proofs to those of \cite{Buchstaber} (Lem 1.2 and latter part of Thm 1.4)
yield $$\Td^\mo(-\Vv^{\rm st}_f)  = \sum_J [M_J]\, \sw_J(-\Vv^{\rm st}_f)$$
for some $[M_J]\in \L^*_{\mo}$ (we do not need any precision on the coefficients).   Hence $\Leftarrow$ follows on applying the Umkehr map. The direction $\Rightarrow$ has been shown  in \S3.
\end{proof}

\vtwo

For oriented bordism, with
$\hh = \MSO^*\ox\Q, \kk = \H^*( \  \ , \L^*_{\mso} \ox\Q) ,  \t = \ch^{\mso},$
and $ \phi_{\mso, \xi}$ the Thom isomorphism defined by the homotopy class of the classifying map $M^\xi\to \MSO_n$, an application of  \thmref{Thm RR} gives in a similar way as for $\MU^*$
\begin{equation*}
\ch^{\mso} (D^{\mso}[f])  = f^{\Hss}_! (\Td^{\,\mso} (-\Vv^{\rm st}_f))
\end{equation*}
 where $\Td^{\,\mso} (\xi) := \phi_{\Hss, -\xi} \,  \ch^{\mso}\,   \phi_{\mso, -\xi} (1) \ \in \H^*( M , \L^*_{\mso} \ox\Q).$
\vtwo
 
Since $\K(M)\ox\Q \to \H^*( M , \L^*_{\mso} \ox\Q), \ \xi\mto \Td^{\,\mso}(\xi),$ is a characteristic class -- insofar as it satisfies Whitney sum additivity \eqref{T hom}  and pulls-back functorially (because Thom isomorphisms and the Chern-Dold character do so) -- then, by uniqueness of the Pontryagin classes, $\Td^{\,\mso} (-\Vv^{\rm st}_f)$ is of the form
\begin{equation}\label{Todd MSO cf}
\Td^{\,\mso} (-\Vv^{\rm st}_f) =  \sum_J b_J\,[M_J]\,\p_J(-\Vv^{\rm st}_f) 
\end{equation}
for some  $[M_J]\in \L^*_{\mso}$ and $b_J\in\Q$.
It thus follows as before that
$$\p^{\mso}_J(D^{\mso}[f])=0  \ \ \forall \ J \<\N^\oo \  \Too  \  D^{\muu} [f]=0 \ \iin  \ \MSO^*(X)\ox\Q,$$
completing the proof of Theorem 1. \\ 

In fact, the form of the expansion  \eqref{Todd MSO cf} can be seen directly using
$\Td^{\mso} (\Vv^{\rm st}_f)  = \Td^{\mso} (\Vv_{e_l\x f}) $
where for $l>>0$ the normal bundle $\Vv_{e_l\x f}$ is a stable representative and setting $\Vv : = \Vv_{e_l\x f}\ox\C$. Let 
$ u_{\mso} \in\tilde{\MSO}^{l-q}(X)$ be its Thom class. 
Define  the top $\MSO$ Conner-Floyd Pontryagin class of $\Vv$ as the Euler class 
\begin{equation*}
\pf^{\mso}_{l-q} (\Vv) := e^{\mso}(\Vv) := i_{\mso}^* u_{\mso}.  
\end{equation*}
Using this it is known  how to define a total $\MSO$ Conner-Floyd Pontryagin class $\pf^{\mso}(\Vv) = \sum_{j\geq 0} \pf^{\mso}_j(\Vv)$ in a natural way, via  Grothendieck's construction, such that  $\pf^{\mso}(\Vv\oplus \Vv\pr) = \pf^{\mso}(\Vv)\p^{\mso}(\Vv\pr)$ and $\pf^{\mso}(\Vv \oplus \ul n) = \pf^{\mso}(\Vv)$ over $\Q$, see \cite{Conner-Floyd2}, \cite{Quillen}.
 We have 
\begin{eqnarray*}
i_! \ch^{\mso}(\pf^{\mso}_{l-q}(\Vv^{\rm st}_f))  = i_! \ch^{\mso}(\p^{\mso}_{l-q}(\Vv)) &: = & i_!\ch^{\mso}(i_{\muu}^* u_\Vv ) \\ & = & i_! i^*_{\Hss} \,\ch^{\mso}(u_\Vv) \\
			  		   & = & \ch^{\mso}(u_\Vv) \cup u_\Vv \\	
					    & = & \phi_{\Hss} \phi_{\Hss}\ii (\ch^{\mso}(u_\Vv) )\cup u_\Vv \\	
					     & = & i_! (\Td^{\mso} (\Vv)) \cup u_\Vv \\			
 						& = & i_! (\Td^{\mso} (\Vv) \cup i^* u_\Vv )\\
						& = & i_! (\Td^{\mso} (\Vv) \cup \p_{l-q}(\Vv)),
\end{eqnarray*}
and so $\ch^{\mso}(\pf^{\mso}_{l-q}(\Vv)) = \Td^{\mso} (\Vv) \cup \p_{l-q}(\Vv)$. The result now follows from an expansion of the desired form of $ \ch^{\mso}(\pf_{l-q}(\Vv))$, which in turn holds as a consequence of the  splitting principle and the  expansion of $ \ch^{\mso}(\pf^{\mso} _1(\Vv))$  in classical Pontryagin classes, as in the complex case from  \eqref{Chern-Dold MU}, a fact which  is  a consequence of the formal group law on $\MSO^*$ (or more easily on $\MSp^*$) as in \cite{Buchstaber}.

\section{Proof of Theorem 2.} 

An oriented cobordism genus is a ring homomorphism 
$\MSO^*(X)\ox\Q \to  \H^*(X, \mathbb{Q}), $  and a complex genus is a ring homomorphism 
$\MU^*(X)\ox\Q \to  \H^*(X, \mathbb{Q}).$  
Just as in the classical case such objects are seen to be defined by multiplicative sequences. Recall, for the latter  one considers the algebra  $\mathcal{H}[z]= 1 + \sum_{j\leq 1}z^j \H^j(M, \mathbb{Q})$
of  polynomials $b(z)=1 + b_1z + b_2z^2 + \dots$ in a formal variable $z$ with  $b_j \in \H^j(M, \mathbb{Q})$, and an algebra endomorphism $\kappa$ of $\mathcal{H}[z]$. If the coefficients $\kappa_j(b_1, \dots, b_j)$ of 
\begin{align*}
\kappa(b(z))= 1 + \kappa_1(b_1)z + \kappa_2(b_1, b_2)z^2 + \dots 
\end{align*}
are polynomials, homogeneous relative to $b_j$ being assigned degree $j$, they are said to define a multiplicative sequence in view of the endomorphism property  of $\kappa$ 
\begin{align}\label{multiplicativityofendom}
\kappa(a(z)\cdot b(z))=\kappa(a(z))\cdot \kappa(b(z)).
\end{align}
A multiplicative sequence determines and is determined by a power series $g \in \mathbb{C}[[z]]$ via the correspondence $\kappa(1+z) = g(z)$ and 
$1+ b_1z + b_2z^2 + \dots=\Pi_{i=1}^M(1 + \beta_iz)$ with $ \beta_i \in \H^{k_i}(M, \mathbb{Q})$, from which (\ref{multiplicativityofendom}) gives
$\kappa(b(z))=\Pi_{i=1}^Mg(\beta_iz),$
identifying $b_j$ with the $i^\text{th}$ elementary symmetric function of $\beta_1, \dots, \beta_M$. On a closed $4m$-dimensional manifold $Y$ this is applied to its total Pontryagin class $1 + p_1 + \cdots + p_m$  to define the oriented genus 
$
\Phi_\kappa(Y)=\int_Y \kappa(p)= \int_Y \kappa_m(p_1, \dots, p_m) \ \in \Q.
$
Similarly, multiplicative sequences in Chern classes lead to complex genera. \\

Vertical genera are constructed in  the same way but using the Pontryagin class of the vertical stable normal bundle and integration over the fibre. Precisely: 

\begin{theorem}\label{vertical genera}
For each rational multiplicative sequence $\{\kappa_m\}$, in the notation of \eqref{pJ}, the correspondence 
\begin{eqnarray}\label{ring isom}
\om \to  \kappa([f^\om])&=& f^\om_! (\kappa(p(-\Vv^{\rm st}_{f^\om})))\nonumber\\[2mm]
&=&\sum_m\underbrace{f^\om_! \left(\kappa_m \left(p_1(-\Vv^{\rm st}_{f^\om}), \dots, p_m(-\Vv^{\rm st}_{f^\om})\right)\right)}_{\in \H^{4m-q}(X, \mathbb{Q})},\nonumber\\
\end{eqnarray}
 defines a vertical genus $ \MSO^*(X) \to  \H^*(X, \mathbb{Q})$
natural with respect to pull-back.
\end{theorem}

\vthree

Analogous statements hold for vertical complex genera in terms of Chern classes and vertical unoriented  genera $\MO^*(X) \to  \H^*(X, \Z_2)$ in terms of vertical Stiefel-Whitney classes.\\ 

\begin{proof} By \thmref{pJf} the map \eqref{ring isom} is a  vertical cobordism invariant.
From the commutativity of  \eqref{MSOduality}, the ring product of $\om^f = D^{\mso}[f], \,\om^g =  D^{\mso}[g] \in\MSO^*(X)$  is the element 
$$\om^{f\xX g} = D^{\mso}[f\xX g] \in\MSO^*(X).$$ 
From the identity $\Vv^{\rm st}_{f \times_{{\tiny X}} g}  = \mu^* \Vv^{\rm st}_f + \nu^* \Vv^{\rm st}_g$ 
of \propref{induced orientations}  for   transverse maps $f:M\to X$ and $g:N\to X$ defining a commutative diagram 
$$\begin{tikzcd}
\MxN \arrow{r}{\mu} \arrow[swap]{dr}[above]{\hfive \    f \times_{{\tiny X}} g}\arrow{d}{\nu} &  M\arrow{d}{f} \\
N \arrow{r}{g}   & X,
\end{tikzcd}$$
we have 
\begin{equation}\label{pontryagin product}
\p(-\Vv^{\rm st}_{f \times_{{\tiny X}} g})  = \mu^* \p(-\Vv^{\rm st}_f) \cup \nu^*\p(-\Vv^{\rm st}_g)
\end{equation}
in $\H^*(M\times_{{\tiny X}} N)$ and hence
\begin{equation*}
\kappa(\p(-\Vv^{\rm st}_{f \times_{{\tiny X}} g}))  = \mu^* \kappa(\p(-\Vv^{\rm st}_f)) \cup \nu^* \kappa(\p(-\Vv^{\rm st}_g)).
\end{equation*}
From $f \times_{{\tiny X}} g = g\circ \nu$  
we have for $a\in \H^*(N,\Q), b\in \H^*(M,\Q)$
\begin{align*}
(f \times_{{\tiny X}} g)_! (\mu^* a\cup \nu^*  b) & = g_!\circ \nu_!(\mu^* a \cup \nu^* b)\\
& = g_! (\nu_! \mu^*a \cup  b)\\
& = g_!(g^*f_!a  \cup  b)\\
& = f_! a \cup g_! b
\end{align*}
using \eqref{reciprocity} for the third equality. Consequently,

\begin{align*}
(f \times_{{\tiny X}} g)_!(\kappa(\p(-\Vv^{\rm st}_{f \times_{{\tiny X}} g})))  
&= (f \times_{{\tiny X}} g)_!  (\mu^* \kappa(\p(-\Vv^{\rm st}_f)) \cup \nu^*\kappa(\p(-\Vv^{\rm st}_g)))  \\[2mm]
& = f_!(\kappa(-\p(\Vv^{\rm st}_f))) \cup g_!(\kappa(\p(-\Vv^{\rm st}_g))).
\end{align*}
\eqref{ring isom} hence defines a ring homomorphism, and so a vertical genus. \\

That the genus  \eqref{ring isom} is natural for pull-back is the equality 
$$ g^* (f_! (\kappa(p(-\Vv^{\rm st}_f))))= \mu_! (\kappa(p(- \Vv^{\rm st}_\mu))).$$ 
But by \eqref{reciprocity} we have $g^* \circ f_! = \mu_!\circ \nu^*$ from which the above identity follows on recalling 
 $\mu^*\Vv^{\rm st}_f = \Vv^{\rm st}_\nu$ and the functoriality of the Pontryagin classes. 
\end{proof}

\vthree 

\begin{ex}[Vertical $\hat{A}$-genus on fibrations]

Consider fibre bundles   $\pi: M \to  X$ and  $\pi': M' \to  X$ , with product 
\begin{equation*}
\begin{tikzcd}
M \times_X M'  \arrow{r}{b'} \arrow{dr}[dashed]{\pi\times_X \pi'} \arrow{d}{b} & M' \arrow{d}{\pi'} \\
M \arrow{r}{\pi}   & X.
\end{tikzcd}
\end{equation*}

Let $p^\pi= 1 + p_1^\pi + \dots$ be the Pontryagin class of the stable normal bundle  $-\Vv^{\rm st}_\pi $ of $\pi$. Since $\pi$ is a submersion there is a (non canonical) vector bundle isomorphism 
\begin{equation*}
T_M \cong T_\pi  + \pi^* T_X
\end{equation*}
where $T_\pi := \Ker (d\pi) $ is the tangent bundle along the fibres (the vertical tangent bundle). Thus 
  $$-\Vv^{\rm st}_\pi  = T_\pi$$
is an honest vector bundle and a vertical orientation is an orientation on each fibre $M_x$ in the usual sense, while  $ q = \dim \, M_x$. 
The $J^{{\rm th}}$ vertical Pontryagin class  is 
\begin{align*}
\p_J(\pi) :=\pi_!(\p_J(T_\pi)) \in \H^*(X, \mathbb{Q}).
\end{align*}
The $\hat{A}$-class on $M$ of  $-\Vv^{\rm st}_\pi  = T_\pi$ is a sum
\begin{align*}
\hat{A}^\pi = 1+ \hat{A}^\pi_4 + \hat{A}^\pi_8 + \dots \in \H^{4*}(M, \mathbb{Q})
\end{align*}
with 
\begin{align*}
\hat{A}^\pi_4 = - \frac{p_1^\pi}{24}, \qquad \hat{A}^\pi_8=\frac{7(p_1^\pi)-4p_2^\pi}{5760}, \qquad \dots 
\end{align*}
giving the vertical $\hat{A}$-genus in $\H^*(X,\Q)$
\begin{align*}
\hat{A}(\pi) :=\pi_!(\hat{A}^\pi) = \underbrace{\hat{A}_{4-q}(\pi)}_{\in \H^{4-q}(X, \mathbb{Q})} + \underbrace{\hat{A}_{8-q}(\pi)}_{\in \H^{8-q}(X, \mathbb{Q})} + \dots\ \ ,
\end{align*}
where
\begin{equation}\label{A8k-q}
\hat{A}_{4-q}(\pi)=\pi_!\left(- \frac{p_1^\pi}{24}\right), \qquad \hat{A}_{8-q}(\pi)=\pi_!\left(\frac{7(p_1^\pi)^2-4p_2^\pi}{5760}\right), \qquad \dots\ \ . 
\end{equation}
Thus
\begin{align*}
\hat{A}(\pi)\cdot\hat{A}(\pi)=\underbrace{\hat{A}_{4-q}(\pi)\hat{A}_{4-q'}(\pi')}_{\in \H^{8-(q+q')}(X, \mathbb{Q})} + \underbrace{\hat{A}_{4-q}(\pi)\hat{A}_{8-q'}(\pi')+ \hat{A}_{8-q}(\pi')\hat{A}_{4-q}(\pi)}_{\in \H^{12-(q+q')}(X, \mathbb{Q})} + \dots 
\end{align*}
which by (\ref{(0.22)}) is expected to coincide with
\begin{align*}
\hat{A}(\pi \times_X \pi')= \underbrace{\hat{A}_{8-(q+q')}(\pi \times_X \pi')}_{\in \H^{8-(q+q')}(X, \mathbb{Q})} + \underbrace{\hat{A}_{12-(q+q')}(\pi \times_X \pi')}_{\in \H^{12-(q+q')}(X, \mathbb{Q})}+ \dots \ \  . 
\end{align*}
To check the first of these equalities $\hat{A}_{8-(q+q')}(\pi \times_X \pi')= \hat{A}_{4-q}(\pi)\hat{A}_{4-q'}(\pi')$ note
\begin{align}\label{(1.23)}
\hat{A}_{4-q}(\pi)\hat{A}_{4-q'}(\pi')= \frac{1}{576}p_1(\pi)p_1(\pi')
\end{align}
from \eqref{A8k-q}, and 
\begin{align*}
\hat{A}_{8-(q+q')}(\pi\times_X \pi') = (\pi \times_X \pi')_!\left( \frac{7(p_1^{\pi \times_X \pi'})^2-4p_2^{\pi \times_X \pi'}}{5760}\right).
\end{align*}
But 
\begin{align*}
(p_1^{\pi \times \pi'})^2=(b^*p_1^* + (b')^* p_1^{\pi'})^2=b^*(p_1^\pi p_1^{\pi}) + (b')^*(p_1^{\pi'} p_1^{\pi'})+ 2 b^*p_1^\pi (b')^*p_1^{\pi'},
\end{align*}
and since $(\pi \times_X \pi')_!=(\pi \circ b)_!= \pi_! \circ b_!= \pi_!' \circ b_!'$ and $b_!b^*=0$, then 
$$
(\pi\times_X \pi')_!(p_1^{\pi \times \pi'})^2  = 2 \pi_! b_!(b^* p_1^\pi(b')^*p_1^{\pi'})
 = 2 p_1(\pi) p_1(\pi')
$$
by the same steps as in the proof of \thmref{vertical genera}. On the other hand, from \eqref{pontryagin product}
\begin{align*}
p_2^{\pi \times_X \pi'}= b^* p_2^\pi + (b')^*p_2^{\pi'} + b^* p_1^\pi (b')^* p_1^{\pi'},
\end{align*}
so
\begin{align*}
(\pi \times_X \pi')(p_2^{\pi \times_X \pi'})= \pi_! b_! (b^* p_1^\pi (b')^*p_1^{\pi'})= p_1(\pi)p_1(\pi').
\end{align*}
Thus
\begin{align*}
(\pi \times_X \pi')_!\left( \frac{7(p_1^{\pi \times_X \pi'})-4p_2^{\pi \times_X \pi'}}{5760}\right) & = \frac{7 \times 2p_1(\pi)p_1(\pi') - 4 p_1(\pi)p_1(\pi')}{5760}\\
& =\frac{1}{576}p_1(\pi)p_1(\pi'),
\end{align*}
which is \ref{(1.23)}.
\end{ex}

\vskip 8mm

{\small \textsc{Department of Mathematics, King's College
London.}}


\begin{thebibliography}{9}

\bibitem[Adams]{Adams}  J. F. Adams, `Stable Homotopy and Generalised Homology', Chicago Lectures in Mathematics, University of Chicago Press, 1974. 

\bibitem[Atiyah1]{Atiyah1}  M. F. Atiyah, `Bordism and cobordism', Proc. Camb. Philos. Soc. 57 (1961), 200-208. 

\bibitem[Atiyah2]{Atiyah2}  M. F. Atiyah, `The signature of fibre-bundles', in Global Analysis (Papers in Honor of K. Kodaira), Univ. Tokyo Press, Tokyo (1969), 73–84.

\bibitem[Atiyah, Hirzebruch]{AtHi} M. F. Atiyah, F. Hirzebruch, `Vector bundles and homogeneous spaces', 1961, Proc. Sympos. Pure Math., Vol. III pp. 7-38 AMS.

\bibitem[Buh\v staber]{Buchstaber} `The Chern-Dold character in cobordism theory I', Mat. Sb. 83 (1970), 575-595.

\bibitem[Cohen, Klein]{Cohen}  R. Cohen, J. Klein, `Umkehr maps', Homology, Homotopy, and Applications, vol. 11, (2009), 17-33.

\bibitem[Cohen, Jones]{CoJo}  R. Cohen, J. Jones, `A homotopy theoretic realization of string topology', Math. Ann. 324, 773–798, 2002.

\bibitem[Conner, Floyd1]{Conner-Floyd1}  P. E. Conner,  E. E. Floyd,  'Differentiable periodic maps'. Bull. Amer. Math. Soc. 68 (1962), no. 2, 76--86.

\bibitem[Conner, Floyd2]{Conner-Floyd2}  P. E. Conner and E. E. Floyd, `The Relation of Cobordism to K-Theories', Lecture Notes in Mathematics, No. 28, Springer-Verlag, Berlin, 1966.

\bibitem[Crabb, James]{CrJa}  M. C. Crabb and I. M. James, `Fibrewise Homotopy Theory',  Springer Monographs in Mathematics, Springer, 1998.

\bibitem[Dold]{Dold}  A. Dold, `Lectures on Algebraic Topology', 2nd ed., Grundlehren 200, Springer, 1980.

\bibitem[Dyer]{Dyer}  E. Dyer, `Relations between cohomology theories', Colloq. Algebraic Topology, Inst. Math. Aarhus Univ. (1962) 89-93.

\bibitem[Hopkins]{Hopkins} `Algebraic topology and modular forms', Proceedings of the ICM, Beijing 2002, vol. 1, 283–309.


\bibitem[Crabb, James]{CrJa}  M. C. Crabb and I. M. James, `Fibrewise Homotopy Theory',  Invent. 146 (2001), 1998.

\bibitem[Kochman]{Kochman}  S. O. Kochman. `Bordism, Stable Homotopy and Adams Spectral Sequences', vol. 7  Fields Institute
Monographs. AMS, 1996.

\bibitem[Lang]{Lang}  S. Lang, `Differential and Riemannian manifolds', Springer-Verlag, 1995.

\bibitem[Lawson, Michelson]{LaMi}  H. B. Lawson, Jr. and M.-L. Michelson, `Spin geometry', Princeton Math. Series 38, 1989.

\bibitem[Joyce]{Joyce} D. Joyce, `On manifolds with corners',  225-258 in S. Janeczko, J. Li and D.H. Phong, editors, Advances in Geometric Analysis, Advanced Lectures in Mathematics 21, 2012. 

\bibitem[Quillen]{Quillen} D Quillen, `Elementary proofs of some results of cobordism theory using Steenrod operations', Advances in Math. 7 (1971), 29 - 56.

\bibitem[Ranicki]{Ranicki} A. Ranicki,  `Algebraic and Geometric Surgery', Oxford Mathematical Monographs, OUP, 2002.
 
\end{thebibliography}
\end{document}